\documentclass[12pt]{amsart}
\usepackage{preamble}

\begin{document}
\title{Categories of Games and their Fra\"iss\'e Theory}
\author[M. Duzi]{Matheus Duzi$^1$}
\address{Instituto de Ci\^encias Matem\'aticas e de Computa\c c\~ao, Universidade de S\~ao Paulo\\
	Avenida Trabalhador s\~ao-carlense, 400,  S\~ao Carlos, SP, 13566-590, Brazil}
\email{matheus.duzi.costa@usp.br}
\thanks{$^1$Supported by FAPESP (2019/16357-1 and 2021/13427-9)}
\author[P. Szeptycki]{Paul Szeptycki}
\address{Department of Mathematics and Statistics, Faculty of Science, York University\\
	Toronto, Ontario, Canada M3J 1P3}
\email{szeptyck@yorku.ca}
\author[W. Tholen]{Walter Tholen}
\address{Department of Mathematics and Statistics, Faculty of Science, York University\\
	Toronto, Ontario, Canada M3J 1P3}
\email{tholen@yorku.ca}
		\subjclass[2020]{Primary 91A44, 18A05;
					Secondary 18A35, 37B05}
\begin{abstract}
	Relying on recent generalizations of the Fra\"iss\'e theory to a broader category-theoretic context, we study the class of abstract finite games played between two players and show the existence of an infinitetly countable game which is ultrahomogeneous and universal with respect to said class. Certain peculiarities of our game categories which clash with the usual framework found in the literature then lead us to formulate weaker category-theoretic properties which still yield a universal and ultrahomogeneous Fra\"iss\'e limit, thus further generalizing the categorical framework for a Fra\"iss\'e theory.  

    \smallskip\smallskip
    \noindent \textit{Keywords:} infinite games, finite games, Fra\"{i}ss\'{e} theory, category theory, model theory, topological dynamics, Ramsey theory.
\end{abstract}

\maketitle	
\pagenumbering{arabic}

\section{Introduction}
The study of the combinatorics of Fra\"iss\'e classes, started in \cite{Fraisse1954}, has received a 
great deal of attention, especially from the perspective of the Pestov-Kechris-Todorcevic correspondence established in \cite{Kechris2005}. Our goal in this work is to study Fra\"iss\'e classes related to some recently introduced categories of infinite games. 

Fra\"isse theory is classically set in a model-theoretic context. However, in \cite{Kubis2014}, Fra\"iss\'e classes have been studied in a broader categorical framework. In light of the categorical theory of mathematical games introduced and studied in \cite{Duzi2024}, we consider here approaches to developing the Fra\"iss\'e theory of games from both a model-theoretic and categorical perspective, with an emphasis on the later (for reasons made apparent below).

In Section \ref{SEC_ModelTheoretical} we recall the basic concepts of Fra\"ss\'e theory in the classical model-theoretic context, introduce the finite objects that we wish to consider (namely, the proper finite games) and describe an example that shows that the model-theoretic approach cannot be applied to our class of finite games.

Section \ref{SEC_Categorical} applies the categorical framework developed in \cite{Kubis2014} and presents a Fra\"iss\'e sequence of finite games and its (direct limit, our candidate for the) Fra\"iss\'e limit, $G_\FL$.

In Section \ref{SEC_Ultrahom} we present a modification of the model-theoretic framework that addresses the shortcomings of the latter approach, culminating in:

\begin{thm*}[Theorem \ref{THM_GamesFL}]
    $G_{\FL}$ is the unique (up to isomorphism) game which:
         \begin{itemize}
             \item is \emph{countable} (in the sense that its game tree and payoff set are countable),
             \item contains a copy of every finite game, 
             \item is ultrahomogeneous w.r.t. $(\emb,\FinGame)$.
         \end{itemize}
\end{thm*}

We return in Section \ref{SEC_WeaklyFiniteSmall} to the categorical framework and introduce the notion of \emph{weak finite smallness}, highlighting its key role in ultrahomogeneity of the Fra\"iss\'e limit in the classical context, characterizing which games satisfy such property and thus showing that, despite our Fra\"iss\'e limit of finite games being ultrahomogeneous (as seen in the previous section), not every finite game is weakly finitely small. 

We single out in Section \ref{SEC_Directedness} two categorical properties weakening weak finite smallness for all small objects that, together, guarantee ultrahomogeneity of the Fra\"iss\'e limit (thus, providing a generalization to the categorical Fra\"iss\'e theory).

At last, relying on the characterization shown in Section \ref{SEC_WeaklyFiniteSmall} and on a KPT-correspondence, we show that, provided with a certain topology, the automorphism group of our Fra\"iss\'e limit of finite games is not extremely amenable and that the class of finite games which are trivial for one of the two players does not satisfy a Ramsey-type combinatorial property.

Our terminology and notation are quite standard. We denote the empty sequence by $\seq{\,}$. Given a finite sequence $t$ with domain $n\in \omega$ (the empty sequence included, with $n=0$), we denote by $|t|=n$ the \textit{length} of the sequence $t$, and, for each $k\le |t|$, $t\upharpoonright k$ be its initial segment with length $k$ (the \textit{truncation} of $t$ up to its $k$-th element). 

Given two finite sequences $t=\seq{x_i: i\le n}$ and $s=\seq{y_i:i\le m}$, we denote by $t^\smallfrown s$ the \textit{concatenation} of $t$ with $s$, that is, 
\[
t^\smallfrown s = \seq{x_0, \dotsc, x_n, y_0, \dotsc, y_m}.
\]
Furthermore, we will write $t^\smallfrown x$ instead of  $t^\smallfrown \seq{x}$ and we let $\seq{x: i< n}$ denote the $x$ constant sequence of length $n$. 

For two distinct infinite sequences $R,R'$, let
\[
\Delta(R,R')=\min\set{n<\omega: R(n)\neq R'(n)}.
\]

Our terminology and notation from category theory are also standard and can be found in, e.g., \cite{MacLane1978} For example, given a category $\mathbf C$ with objects $x,y$, we denote by $\mathbf{C}(x,y)$ the class of all morphisms from $x$ to $y$. And without any risk of ambiguity in our setting, we write $\mathbf C\subseteq \mathbf D$ when $\mathbf C$ is a full subcategory of $\mathbf D$.

For a functor $\mathrm{F}\colon \mathbf{C}\to \mathbf{D}$, we denote by $\colim\mathrm{F}$ its colimit object (unique up to isomorphism) in $\mathbf{D}$, if it exists. For a family of objects $\seq{A_i:i\in I}_{i\in I}$ in a category $\mathbf{C}$, we denote its coproduct by $\coprod_{i\in I} A_i$ (if it exists). For the sake of clarity, we denote the disjoint union of a family of sets $\set{X_i:i\in I}$ as $\bigsqcup_{i\in I}X_i$.

We will often regard the ordinal $\omega$ as a category: its objects are the natural numbers and there is a (unique) morphism $ n\to m$ if, and only if, $n\le m$.

Let us now recall some game-theoretic notions from  \cite{Duzi2024} which we will be using later:

\begin{defn}\label{DEF_Game}
	For a set $M$, a set $T\subseteq M^{<\omega}=\bigcup_{n<\omega}{M^n}$ of finite sequences in $M$ is a \textit{game tree over $M$} if 
	\begin{itemize}
		\item[(\rom{1})] For every $t\in T$ and $k\le|t|$, one has $t\restrict k\in T$;
		\item[(\rom{2})] For every $t\in T$, there is an element $x$ such that $t^\smallfrown x\in T$.
	\end{itemize}	
	We say that $T$ is a \textit{game tree} if $T$ is a game tree over some set $M$ and we write
 \[\M(T) = \set{x: t^\smallfrown x \in T \text{ for some } t\in T}.\]

        A pair $G = (T, A)$ is an \textit{infinite game} if $T$ is a game tree over a set $M$ and $A$ is a subset of 
	\[
	\Run(T) = \set{R\in M^{\omega}: R\restrict n\in T \text{ for every $n<\omega$}}.
	\]
\end{defn}

We write $G\le G'$ for $G=(T,A)$ and $G'=(T',A')$ when $T\subseteq T'$ and $A = A'\cap \Run(T)$.

We refer to the elements of $T$ as {\em moments} of the game, to the elements of $\Run(T)$ as \emph{runs} of the game, to the set $A$ in $G=(T,A)$ as the \emph{payoff set} of the game and we say that a run $R\in\Run(T)$ is \emph{won by $\ali$} if $R\in A$ (otherwise, we say that such run is won by $\bob$). 

\begin{defn}[Chronological map]\label{DEF_fChronological}
	\sloppy For game trees $T_1$ and $T_2$, a mapping $f\colon T_1\to T_2$ is \emph{chronological} if $f$ preserves length and truncation of moments; that is: if for every $t\in T_1$, $|f(t)|=|t|$ and $f(t\restrict k)=f(t)\restrict k$ for all $k\le |t|$.

        We say that $f\colon T_1\to T_2$ is a {\em chronological embedding} if it is an injective chronological map. 
\end{defn}  

As pointed out and extensively used in \cite{Duzi2024}, any chronological map $f\colon T_1\to T_2$ extends to a mapping $\overline{f}\colon \Run(T_1)\to \Run(T_2)$, uniquely determined by  $\overline{f}(R)\restrict n=f(R\upharpoonright n)$ for every $n<\omega$; thus we write
\[
\overline{f}=\colim f_n,
\]
with $f_n: T_1(n)\to T_2(n)$ denoting the indicated restriction of $f$.

\begin{defn}[Game morphisms]\label{DEF_Amorph}
	Let $G_1=(T_1,A_1)$ and $G_2=(T_2,A_2)$ be games and $f\colon T_1 \to T_2$ be chronological. Then 
	\begin{itemize}
		\item[(A)] $f$ is an $\A$-\textit{morphism} if $\overline{f}(R)\in A_2$ for every run $R\in A_1$ (i.e., $\overline{f}(R)$ is won by $\ali$ in $G_2$ whenever $R$ is won by $\ali$ in $G_1$) and
		\item[(B)] $f$ is a $\B$-\textit{morphism} if $\overline{f}(R)\in \Run(T_2)\setminus A_2$ for every run $R\in \Run(T_1)\setminus A_1$ (i.e., $\overline{f}(R)$ is won by $\bob$ in $G_2$ whenever $R$ is won by $\bob$ in $G_1$).
	\end{itemize}	
 We say that $f\colon G_1 \to G_2$ is a {\em game embedding} if it is injective and both an $\A$- and a $\B$-morphism, in which case we may write $f\colon G_1  \embed G_2$.
\end{defn} 

The main categories which will be analyzed throughout this paper are:

\begin{itemize}
	\item $\Gmes$: 
	\begin{itemize}
		\item objects are game trees,
		\item morphisms are chronological mappings.
	\end{itemize}
	\item $\Games_\A$: 
	\begin{itemize}
		\item objects are games,
		\item morphisms are $\A$-morphisms.
	\end{itemize}
	\item $\Games_\B$: 
	\begin{itemize}
		\item objects are games,
		\item morphisms are $\B$-morphisms.
	\end{itemize}  
        \item $\Games_{\emb}$: 
	\begin{itemize}
		\item objects are games,
		\item morphisms are game embeddings.
	\end{itemize} 
\end{itemize} 

All of these were introduced and studied in \cite{Duzi2024} where, for example, it was shown (see Section 9.1.2 in \cite{Duzi2024}) that the coproduct in $\Gmes$ of a family $\seq{T_j:j\in J}$ of game trees may be given by
\begin{gather*}
	\coprod_{j\in J}T_j=\{\seq{\,}\}\cup\left(\bigsqcup_{j\in J}\left(T_{j}\setminus\{\seq{\,}\}\right)\right)
\end{gather*}
and the coproduct in $\Games_\A$ of a family $\seq{(T_j,A_j):j\in J}$ of games by $(\coprod_{j\in J}T_j, \bigsqcup_{j\in J}A_j)$.

\section{Model-theoretic approach} \label{SEC_ModelTheoretical}
Let us first recall the model-theoretic notion of a Fra\"iss\'e class (we follow \cite{Hodges1993} as reference).

Let $L$ be a countable signature and $D$ an $L$-structure (for the sake of existence of initial objects, we allow empty structures as long as $L$ has no constants). The \textit{age} of $D$ is the class $\mK$ of all finitely generated structures which can be embedded into $D$. We say that $D$ has \textit{countable age} if its age $\mK$ is, up to isomorphisms, countable and we say that $D$ is {\em ultrahomogeneous} if for every finitely generated $L$-structure $A$ and all embeddings $f,g\colon A\to D$ there is an automorphism $h\colon D\to D$ such that $h\circ f = g$. 

Now let $\mK$ be a class of finitely generated $L$-structures. Then we say that $\mK$ has the
\begin{itemize}
	\item Hereditary property (HP for short) if for all $A\in \mK$ and $B$ finitely generated substructure of $A$, $B\in \mK$;
	\item Joint embedding property (JEP for short) if for all $A,B\in \mK$ there is a $C\in \mK$ such that both $A$ and $B$ embed into $C$.
	\item Amalgamation property (AP for short) if for all $A,B,C\in \mK$ and embeddings $f\colon A \embed B$, $g\colon A \embed C$ there is a $D\in \mK$ and there are embeddings $U\colon B \embed D$, $v\colon C \embed D$ such that $u  f=v  g$.
\end{itemize} 

Note that AP implies JEP in the case where the class of all finitely generated $L$-structures has a minimum. 

The main result from \cite{Fraisse1954}, states:

\textit{``Let $L$ be a countable signature and let $\mK$ be a non-empty at most countable set of finitely generated $L$-structures which has the HP, the JEP and the AP. Then there is an $L$-structure $\FL(L)$, unique up to isomorphism, such that 
	\begin{itemize}
		\item $\FL(\mK)$ is at most countable,
		\item $\FL(\mK)$ has age $\mK$,
		\item $\FL(\mK)$ is ultrahomogeneous.''
\end{itemize}}

Any class of $L$-structures that is at most countable (up to isomorphisms), has the HP, the JEP and the AP is called a {\em Fra\"iss\'e class}. 

Now let us get back to games. We are particularly interested in the class of {\em finite} games.

\begin{defn}\label{DEF_FinGames}
	Let $\G=(T,A)$ be a game. We say that $G$ is \textit{finite} if $|\Run(T)|$ is finite.
\end{defn}

The motivation for defining finite games as in \ref{DEF_FinGames} should be clear: since a game with its set of runs being finite has only finitely many outcomes, each player has only finitely many options to choose from at every turn of the game and, from some inning $n<\omega$ onward, both players will have a \textit{single} option -- so that we might as well consider the game to be over at such inning. In other words, our finite games are, in essence, \textit{actual finite games}. 

Moreover, let $\CUMet$ be the category of complete ultrametric spaces with diameter at most $1$ whose morphisms are $1$-Lipschitz maps and then consider $\Sub{\CUMet}$ as the \textit{functor-structured category} (see Definition 5.40 in \cite{Adamek1990}) of the forgetful functor $\mathrm{U}\colon \CUMet\to \Sets$, i.e., the category whose objects are pairs $(X,A)$, with $X$ object in $\CUMet$ and $A\subseteq X$, and whose morphisms $f\colon (X,A)\to (Y,B)$ are $1$-Lipschitz maps $f\colon X\to Y$ such that $f[A]\subseteq B$. Then it was shown in \cite{Duzi2024} that $\Games_\A$ is equivalent to a full coreflective subcategory $\MetGame$ of $\Sub{\CUMet}$, and one easily shows that the finite games defined here are precisely those which are mapped to finite ultrametric spaces by the equivalence presented in \cite{Duzi2024}.

{For a signature $L$, let $\mathbf{S}_L$ be the category whose objects are $L$-structures and morphisms are $L$-embeddings. Let $\Games_{emb}$ be the category whose objects are games and morphisms are game embeddings. We will now show that the classical model-theoretic approach to studying the Fra\"iss\'e Theory of our category of finite games is not possible:
	
	\begin{prop}\label{PROP_NoL}
		There is no signature $L$ and injective functor  $\mathrm{F}\colon \Games_{emb}\to \mathbf{S}_L$ satisfying all the conditions below:
		\begin{itemize}
			\item for every chain of games $\seq{G_n:n<\omega}$, $\mathrm{F}(\bigcup_{n<\omega}G_n)=\bigcup_{n<\omega}\mathrm{F}(G_n)$,
			\item if $G$ is a finite game, then $\mathrm{F}(G)$ is a finitely generated $L$-structures.
		\end{itemize}
	\end{prop}
	
}

\begin{proof} Consider first the following example:

\begin{ex}\label{EX_notFinRep}
	Let
	\[
	c_{00}(\omega) = \set{R\in \omega^\omega: \text{ $R$ is eventually $0$}}=\set{R_n:n<\omega}.
	\] 
	
	Then, for each $n<\omega$, consider 
	\begin{gather*}
		T_n=\set{R_i\restrict m: m\in\omega, i\le n}, \\
		A_n=\set{R_i:i\le n},\\
		\text{and } G_n = (T_n,A_n).
	\end{gather*}
	
	If we let $G_{\FL}=\bigcup_{n\in \omega}G_n=(\bigcup_{n\in \omega}\omega^n, c_{00}(\omega))$, then the constant $\seq{1:n<\omega}\in \Run(\bigcup_{n\in \omega}\omega^n)=\omega^\omega$ is a run in which $\bob$ wins. Hence, $G=(T,\emptyset)\le G_{\FL}$, with $T=\set{\seq{1:i<n}:n<\omega}$. However, $G$ does not embed into any $G_n$, as $\ali$ wins in every run of $G_n$ for every $n<\omega$.
\end{ex}

That this example concludes the proof, can be seen by recalling that if $L$ is a signature, $\set{A_n:n<\omega}$ is a chain of $L$-structures and $X$ is a finitely generated substructure of $\bigcup_{n<\omega}A_n$, then there is an $n<\omega$ such that $X$ is a finitely generated substructure of $A_n$.

\end{proof}

As mentioned above, the main conclusion from Proposition \ref{PROP_NoL} is that a direct application of the classical model-theoretic approach to studying the Fra\"iss\'e Theory of our category of finite games is not possible. Of course, denoting $(\bigcup_{n\in \omega}\omega^n, c_{00}(\omega))$ by $G_\FL$ is suggestive: this will be an analogous Fra\"iss\'e limit that we will obtain later on, so we still have some options left.

\section{Category-theoretic approach}\label{SEC_Categorical}
In \cite{Kubis2014}, Kubi\'s presents a categorical framework for the study of Fra\"iss\'e classes and their limits which we now apply to our categories of games.

The use of the term ``\textit{limit}'' in ``\textit{Fra\"iss\'e limit}'', within the model-theoretic context, is due to the fact that said object can be obtained by taking the colimit of a cleverly constructed tower 
\[
D_0 \embed D_1 \embed \cdots \embed D_n \embed \cdots 
\]

It is then based on the construction of said tower that the notion of ``\textit{Fra\"iss\'e limit}'' is generalized to a broader categorical context in \cite{Kubis2014} -- to better suit the work done in later sections, we shall apply it within the following general setting: while working within an ambient category $\bC$, we select a class $\mM$ of $\bC$-morphsisms (meant to play the role of ``embedding-like'' morphisms, although the framework provided in \cite{Kubis2014} can be used in broader contexts, such as that of Projective Fra\"iss\'e Theory) and a full subcategory $\bF$ of $\bC$ (meant to play the role of selecting the ``finite-like'' objects which we are interested in). We will also assume (unless stated otherwise) that $\bC$ is closed under colimits of $\omega$-sequences in $\bF_\mM$ (which denotes the category whose objects coincide with $\bF$ and whose morphisms coincide with those in $\bF$ which are also in $\mM$). In this case, it should be clear that the construction of the aforementioned tower should take place in $\bF_\mM$:

\begin{defn}\label{DEF_FraisseSeq}
	A sequence $\mathrm{F}\colon \omega \to \mathbf{F}_\mM$ is called a \textit{Fra\"iss\'e sequence} if the following conditions are satisfied:
	\begin{itemize}
		\item[(U)] for every object $x$ in $\mathbf{F}$ there is an $n<\omega$ such that $\mM(x,\mathrm{F}(n))\neq\emptyset$,
		\item[(A)] for every $n<\omega$ and for every morphism $f\colon \mathrm{F}(n) \to x$ in $\mathbf{F}_\mM$ there is an $m\ge n$ and a morphism $g\colon x\to \mathrm{F}(m)$ in $\mathbf{F}_\mM$ such that $\mathrm{F}_n^m=g  f$ (where $\mathrm{F}_n^m$ denotes the morphism $\mathrm{F}(n\le m)$).
	\end{itemize}
\end{defn}

Thus, our goal here is to show that, even though Proposition \ref{PROP_NoL} tells us that we cannot explore any meaningful \emph{Fra\"iss\'e Theory} of finite games through the lens of model theory, we can do so through the aforementioned category-theoretic approach.

But first, let us recall a few generalizations of the basic concepts seen in Section \ref{SEC_ModelTheoretical}.

\begin{defn}
   We say that $\mathbf{F}_\mM$ has the \emph{joint embedding property} (JEP for short) if for all pairs of objects $A,B$ in $\mathbf{F}$ there are morphisms $a\colon A\to C$ and $b\colon B\to C$ in $\bF_\mM$.

    We say $\mathbf{F}_\mM$ satisfies the \emph{amalgamation property} (AP for short) if for all $f\colon A\to B$ and $f'\colon A\to B'$ in $\mathbf{F}_\mM$ there are $g\colon B\to C$ and $g'\colon B'\to C$ in $\mathbf{F}_\mM$ such that $g'\circ f' = g\circ f$.
\end{defn}

It goes without saying that JEP holds if $\bF_\mM$ has binary coproducts, with coproduct injections lying in $\mathcal M$; and that the condition on the coproduct injections is necessary for JEP if the class $\mathcal M$ is left cancellable (that is, if $f\circ i\in \mathcal M$ always implies $i\in\mathcal M$). Likewise, AP may be verified using a pushout of the given span of morphisms in $\mathcal M$, provided that the pushout injections again belong to $\mathcal M$, which is a necessary condition under the right cancellation property of $\mathcal M$. Furthermore, the AP in $\bF_\mM$ is equivalent to the JEP holding in all of the comma categories $(A/\mathbf F)_{\mathcal M}$.

Now, returning to our game-theoretic context: let us denote the class of all game embeddings in $\Games_\A$ by $\mathrm{emb}$. Furthermore, let $\mathbf{Fin}\Games\subset \Games_\A$ be the (full) subcategory of $\Games_\A$ comprised of finite games. 

Since $\FinGame_{\emb}$ is closed by coproducts and pushouts (for these constructions in $\Games_\A$, we refer to Section 9 in \cite{Duzi2024}), it follows that $\FinGame_\emb$ satisfies the JEP and the AP. 

We also need the following variation of the well known concept of \emph{injective object} (see Section 6 of \cite{Kiss1983}).

\begin{defn}
    Let $\mM$ be a class of morphisms in a category $\mathbf{C}$. We say that an object $X$ in $\mathbf{C}$ is injective with respect to $\mathcal M$ if every $f\colon A\to X$ in $\bC$ can be extended through any $g\colon A\to B$ in $\mM$.
\end{defn}

\begin{thm}\label{THM_GFL_Injective}
    The game $G_{\FL}$ defined in Example \ref{EX_notFinRep} is injective in $\Games_{\emb}$ w.r.t. the class of all morphisms in $\FinGame_\emb$.
\end{thm}
\begin{proof}
Item (a) in the following lemma is essentially the proof, but we will need item (b) later, in Theorem \ref{THM_FraisseSeq}: 

\begin{lemma}\label{LEMMA_GFL_Age}
Let $G\le G'$ be finite games and  $f\colon G\embed G_{\FL}$. Then there is an embedding $\tilde{f}\colon G'\embed G_{\FL}$ such that:
	\begin{itemize}
		\item[(a)] $\tilde{f}$ extends $f$;
		\item[(b)] for all $k<n<\omega$ such that there is a $t^\smallfrown x\in T'\setminus T$ with $\tilde{f}(t^\smallfrown x)=\tilde{f}(t)^\smallfrown n$ there is a $y$ such that $t^\smallfrown y\in T'$ and $\tilde{f}(t^\smallfrown y)=\tilde{f}(t)^\smallfrown k$.
	\end{itemize}
\end{lemma}
\begin{proof}
	Suppose $(T,A)=G\le G'=(T',A')$ are finite and $f\colon G \embed G_{\FL}$ is an embedding.
	
	Then we recursively define $\tilde{f}\colon T'\to \bigcup_{n<\omega}\omega^n$ as  follows.
	
	\begin{itemize}
		\item Obviously, $\tilde{f}(\seq{\,})=\seq{\,}$.
		\item Suppose $\tilde{f}$ is defined up to $T'(n)$ and let $t\in T'(n)$. Then there are only finitely many elements $x$ such that $t^\smallfrown x\in T'$, so let 
		\begin{gather*}
			F=\set{x_0, \dotsc, x_m}=\set{x: t^\smallfrown x\in T'}\text{ and}\\
			H = \set{n<\omega: f(t^\smallfrown x_i)=f(t)^\smallfrown n, i\le m, t^\smallfrown x_i\in T}.
		\end{gather*}
		In this case, we recursively define $\tilde{f}(t^\smallfrown x_i)$ for each $i\le m$ as
		\begin{itemize}
			\item $\tilde{f}(t^\smallfrown x_i)=f(t^\smallfrown x_i)$ if $t^\smallfrown x_i\in T$,
			\item $\tilde{f}(t^\smallfrown x_i)=\tilde{f}(t)^\smallfrown 0$ if $t^\smallfrown x_i\not\in T$, $t\not\in T$ and there is a unique $R\in A'$ extending $t^\smallfrown x_i$,
			\item $\tilde{f}(t^\smallfrown x_i)=\tilde{f}(t)^\smallfrown 1$ if $t^\smallfrown x_i\not\in T$, $t\not\in T$ and there is a unique $R\in \Run(T')\setminus A'$ extending $t^\smallfrown x_i$,
			\item $\tilde{f}(t^\smallfrown x_i)=\tilde{f}(t)^\smallfrown k$ otherwise, where $k=\min (\omega\setminus (H\cup H_i))$ with 
			\[
			H_i = \set{n<\omega: f(t^\smallfrown x_j)=f(t)^\smallfrown n, j\le i}.
			\]
		\end{itemize}
	\end{itemize}	
	
	Then $\tilde{f}$ is clearly injective, chronological and extends $f$. Now, given $R\in \Run(T')\setminus \Run(T)$, note that there is an $n<\omega$ such that $R$ is the unique branch of $T'$ extending $R\restrict n$, so that, for every $m\in\omega$, $\tilde{f}(R\restrict n+m+1)=\tilde{f}(R\restrict n+1)^\smallfrown \seq{i: j< m}$, with 
	\[
	i=\begin{cases}
		0 \text{ if $R\in A$,}\\
		1 \text{ if $R\notin A$.}
	\end{cases}
	\]
	Hence, $\overline{\tilde{f}}(R)\in c_{00}(\omega) \iff R\in A$, which shows that $\tilde{f}$ is an embedding. 
\end{proof}
\end{proof}

\begin{cor}\label{COR_FinGameCountable}
    The class of finite games in $\Games_\A$ is countable (up to isomorphisms).
\end{cor}
\begin{proof}
    This follows from Lemma \ref{LEMMA_GFL_Age}, since it shows that every finite game can be embedded in $G_{\FL}$ and there are only countably many finite subgames of $G_{\FL}$.
\end{proof}

These results are important because we can obtain from them:

\begin{thm}\label{THM_FraisseSeq}
	There is a Fra\"iss\'e sequence $\mathrm{F}\colon \omega \to \FinGame_\emb$ and its colimit in $\Games_\A$ is $G_{\FL}$.
\end{thm}
\begin{proof}
	The existence of the Fra\"iss\'e sequence $\mathrm{F}$ is a consequence of Theorem 3.7 in \cite{Kubis2014}, but we nevertheless present a direct proof in what follows.
	
	Let $\mK$ be the countable set of all finite games (up to isomorphisms), as seen in Corollary \ref{COR_FinGameCountable}, and then
	\[
	\mF=\bigcup_{G,G'\in\mK}\mathbf{F	in}\Games_{\emb}(G',G)=\set{f_k:k\in\omega}
	\]
	(note that $\mF$ is countable because $\mathbf{F	in}\Games_{\emb}(G',G)$$    $ is finite for all finite games $G,G'$).
	
	We recursively define $\mathrm{F}$. Let $\mathrm{F}(0)=\dom(f_0)$ and then suppose $\mathrm{F}(k)$ is defined for all $k\le n$, as well as $\mathrm{F}_k^l$ for all $k\le l\le n$. In this case, using $f_n\colon G'\to G$:
	\begin{itemize}
		\item[(i)] If $G'=\mathrm{F}(k)$ for some $k\le n$, let \begin{tikzcd}[column sep = 5em]
			\mathrm{F}(n) \arrow[r, hook, "\mathrm{F}_n^ {n+1}" above]		&	\mathrm{F}(n+1) \end{tikzcd} be such that the following diagram is a pushout
		\begin{center}
			\begin{tikzcd}[column sep = 3em]
				\mathrm{F}(k) \arrow[d, hook, "f_n"{auto,swap}] \arrow[r, hook, "\mathrm{F}_k^n"{above}] 	&	\mathrm{F}(n)  \arrow[d, hook, "\mathrm{F}(n\le n+1)"{auto}]	\\
				G\arrow[r, hook, "\tilde{f}_n"{below}]	&	\mathrm{F}(n+1).
			\end{tikzcd}
		\end{center}
		\item[(ii)] Otherwise, let \begin{tikzcd}[column sep = 1em]
			\mathrm{F}(n) \arrow[r, hook, "i" above]		&	\mathrm{F}(n)\sqcup G' & G' \arrow[l, hook', "i'" above]\end{tikzcd} be the coproduct diagram and then $\mathrm{F}(n+1)$ be such that the following diagram is a pushout
		\begin{center}
			\begin{tikzcd}
				G' \arrow[d, hook, "f_n"{auto,swap}] \arrow[r, hook, "i'"{above}] 	&	\mathrm{F}(n)\sqcup G'  \arrow[d, hook, "g"{auto}]	\\
				G\arrow[r, hook, "\tilde{f}_n"{below}]	&	\mathrm{F}(n+1).
			\end{tikzcd}
		\end{center}
		Finally, let $\mathrm{F}_n^{n+1} = g  i$.
	\end{itemize}
	
	It should be clear that $\mathrm{F}$ satisfies condition (U) of Definition \ref{DEF_FraisseSeq}, so suppose $f\colon \mathrm{F}(k)\embed G$ is some embedding between finite games and $k\in\omega$. Then we may assume that $f=f_n$ for some $n<\omega$. 
	
	If $k\le n$, then we got case (i) at the $n$th stage of our recursion, so $\tilde{f}_n\colon G \embed \mathrm{F}(n+1)$ attests condition (A).
	
	Otherwise, we are in case (ii) at the $n$th stage of our recursion with $G'=\mathrm{F}(k)$, so $h=\tilde{f}_n  f_n\in \mM(\mathrm{F}(k),\mathrm{F}(n+1))$. On the other hand, $\mathrm{F}_n^k\in \mM(\mathrm{F}(n+1),\mathrm{F}(k))$, thus $\overline{h}\colon \Run(\mathrm{F}(k))\to \Run(\mathrm{F}(n+1))$ must be a surjective isometry (because both $\Run(\mathrm{F}(k)$ and $\Run(\mathrm{F}(n+1))$ are finite) and therefore $h$ is an isomorphism. Hence, $f_n$ must also be an isomorphism and $f_n^{-1}\colon G \embed\mathrm{F}(k)$ attests condition (A), which concludes the proof that $\mathrm{F}$ is a Fra\"iss\'e sequence.
	
	We now show that  $G_{\FL}$ is the colimit of $\mathrm{F}$ by recursively constructing the initial cocone 
	\[
	\set{f_n\colon \mathrm{F}(n)  \embed G_{\FL}:n<\omega}.
	\] 
	
	\begin{itemize}
		\item First, let $f_0\colon \mathrm{F}(0)  \embed G_{\FL}$ be the extension of $(\emptyset,\emptyset)\embed G_{\FL}$, as obtained in Lemma \ref{LEMMA_GFL_Age}.
		\item Suppose $f_n\colon \mathrm{F}(n) \embed G_{\FL}$ is defined and let $f_{n+1}\colon \mathrm{F}(n+1) \embed G_{\FL}$ be again the extension of $f_n\colon \mathrm{F}(n)  \embed G_{\FL}$ as obtained in Lemma \ref{LEMMA_GFL_Age}.
	\end{itemize}
	
	It is clear that $\set{f_n\colon \mathrm{F}(n) \embed G_{\FL}:n<\omega}$ is then a cocone which commutes with the morphisms in the sequence $\mathrm{F}$. Moreover, by property (b) stated in Lemma \ref{LEMMA_GFL_Age}, for every $t\in G_{\FL}$ there is an $n<\omega$ and an $s$ in the decision tree of $\mathrm{F}(n)$ such that $f_n(s)=t$. Thus, it follows from the fact that all $f_n$s are embeddings that such cocone is the colimit of $\mathrm{F}$ (see Lemma \ref{LEMMA_GameCollectivelyE*}).
\end{proof}

We end this section by noting that the colimit in $\Games_{\B}$ of the Fra\"iss\'e sequence $\mathrm{F}$ of Theorem \ref{THM_FraisseSeq} is not $G_{\FL}$. In fact, it can be shown that it is $(\bigcup_{n\in \omega}\omega^n, \omega^\omega\setminus c_{00}(\omega))$, instead. However, since $\Games_\B$ is isomorphic to $\Games_\A$, we will focus exclusively on the latter for the rest of this paper.

\section{Ultrahomogeneity by an analogous approach}\label{SEC_Ultrahom}
We saw in Section \ref{SEC_Categorical} how one constructs the analogue to the tower of ``finite-like'' objects which gives rise to the Fra\"iss\'e limit in the model theoretical context, so that the category $\mathbf{F}_\mM$ is usually comprised of these so called ``finite-like'' objects -- we shall now tackle the matter of computing the actual direct limit of such Fra\"iss\'e sequences in order to obtain a so-called ``Fra\"iss\'e limit'' in a more general categorical context.

We start by noting that, even though we build the tower only with classes of ``finite-like'' structures in the model theoretical context, the Fra\"iss\'e limit (taken as the direct limit of such special towers) of these classes is most often not a finite structure itself -- thus, it is only natural to assume that we will not obtain Fra\"iss\'e limits of a category $\mathbf{F}_\mM$ with a Fra\"iss\'e sequence $\mathrm{F}$ within $\mathbf{F}_\mM$, in most cases. This is why we started Section \ref{SEC_Categorical} by considering $\bF$ and $\mM$ within an ambient category $\mathbf{C}$. 

One may wonder whether we should consider the ambient $\bC$ as comprising only of the morphisms in the desired class $\mM$. However, such approach will often run into problems with the colimit-taking process:

\begin{ex}\label{EX_Met_emb_no_colim}
    Let $\mathbf{Met}$ be the category of metric spaces, whose morphisms non-expanding maps, $\mathbf{FinMet}$ be its full subcategory of finite metric spaces, and $\mathrm{mon}$ be the class of injective non-expanding maps (i.e., its class of monomorphisms). Then it can be easily demonstrated that there are sequences in $\mathbf{FinMet}_{\mathrm{mon}}$ which do not have any colimit cocone in $\mathbf{Met}_{\mathrm{mon}}$ -- however, every sequence in $\mathbf{FinMet}_{\mathrm{mon}}$ \emph{does} have colimit cocone in $\mathbf{Met}$.  
\end{ex}

Therefore, in order to incorporate such a crucial detail to our theory, we will normally be working in the following setting: 

\begin{defn}
    Suppose $\mathbf{F}$ is a full subcategory of $\mathbf{C}$ and $\mathcal{M}$ is a class of morphisms in $\mathbf{C}$ such that every sequence in $\mathbf{F}_{\mathcal{M}}$ has a colimit in $\mathbf{C}$. We then say that an object $F$ is a \emph{Fra\"iss\'e limit of the pair $(\mathbf{F},\mathcal{M})$ in $\mathbf{C}$} if there is a colimit cocone $(f_n\colon \mathrm{F}n\to F)_{n<\omega}$ in $\mathbf{C}$ of a Fra\"iss\'e sequence $\mathrm{F}\colon \omega\to \mathbf{F}_{\mathcal{M}}$ in $\mathbf{F}_\mM$ (here, we note that our chosen terminology extends the model-theoretic terminology by explicitly determining the morphisms which are being considered).
\end{defn}

Nevertheless, even within the above setting, other problems arise: there is no guarantee that the resulting Fra\"iss\'e limit has the desired properties of the classical Fra\"iss\'e limit. Namely:

\begin{defn}\label{DEFN_Ultrahom}
     Suppose $\mathbf{F}$ is a full subcategory of $\mathbf{C}$ and $\mathcal{M}$ is a class of morphisms in $\mathbf{C}$.
    
    We say that an object $U$ in $\mathbf C$ is {\em ultrahomogeneous w.r.t. $(\mathcal{M},\mathbf F)$}, if for every object $A$ in $\mathbf F$ and morphisms $f,g\colon A \to U$ in $\mathcal{M}$ there is an automorphism $u\colon U\to U$ in $\mathbf{C}$ such that $u\circ f = g$.
\end{defn}

We should point out that the concept of ultrahomogeneity as described in Definition \ref{DEFN_Ultrahom} will be found in the literature only with $\mathbf{C}_\mM=\mathbf{C}$ -- we consider the possibility of a strictly smaller class $\mM$ in $\mathbf{C}$ based on the fact that we are only interested to check ultrahomogeneity for embeddings $f,g\colon A \embed U$ and, while it is usually the case that in the model theoretic approach every morphism is an embedding, we are interested in exploring more general category theoretical contexts in which this does not hold. In order to see this, we first explicitly compute the colimits of sequences in $\Games_{\A}$.

Let  ${\mathrm{S}}\colon\omega\to \Games_{\A}$ be a sequence with $\mathrm{S}n=(T_n,A_n)$ and consider $\sim$ as the equivalence relation over $\coprod_{n<\omega}T_n$ with $t\in T_k$ and $s\in T_m$ being such that $t\sim s$ if, and only if, there exists $N\ge k,m$ such that $\mathrm{S}_k^N(t)=\mathrm{S}_m^N(s)$.
	
	Define $G_{\mathrm{S}}=(T_{\mathrm{S}},A_{\mathrm{S}})$ as 
	\begin{gather*}
		T_{\mathrm{S}}=\{\seq{\,}\}\cup\set{\seq{[t\restrict i]:0<i\le |t|}:t\in \coprod_{n<\omega}T_n},\\
		\text{and } A_{\mathrm{S}}=\set{\seq{[R\restrict i]:i\in \omega\setminus\{0\}}: R\in \bigsqcup_{n<\omega}A_n},
	\end{gather*}
	where $[s]$ is the $\sim$-equivalence class of $s\in\coprod_{n<\omega}T_n$.
	
	For each $n<\omega$, let ${q_n}\colon T_n\to T_{\mathrm{S}}$ be such that
	
	   \begin{equation}\label{EQ_GamesA_cocone}
	q_n(t)=\begin{cases}
		\seq{[t\restrict i]:0<i\le n}, \text{ if $t\neq \seq{\,}$}\\
		\seq{\,}, \text{ otherwise.}
	\end{cases}
	\end{equation}
	
	Then it is clear that each $q_n$ is an $\A$-morphism. Moreover, $q_k(t)=q_m(s)$ if, and only if $t\sim s$. In this case:
	
	\begin{lemma}\label{LEMMA_GameASeqLim}
		Let  ${\mathrm{S}}\colon \omega\to \Games_{\A}$ be a sequence with $\mathrm{S}n=(T_n,A_n)$. Then
	   \[
		\seq{{q_n}\colon \mathrm{S}n\to G_{\mathrm{S}}}_{n<\omega}
	   \]
		is the colimit cocone of ${\mathrm{S}}$.
	\end{lemma}
	\begin{proof}
		Suppose $\seq{{f_n}\colon \mathrm{S}n \to G}_{n<\omega}$ commutes with the sequence $\mathrm{S}$. We use the universal property of the coproduct to find the unique ${f}\colon \coprod_{n<\omega}\mathrm{S}n\to G$ commuting with the compositions of $f_n$ with the inclusion of $\mathrm{S}n$ into $\coprod_{n<\omega}\mathrm{S}n$. Then it follows from the fact that $f^{-1}(\{f(t)\})\supseteq [t]$ for every $t\in \coprod_{n<\omega}\mathrm{S}n$ and that $f$ is an $\A$-morphism that ${\tilde{f}}\colon G_{\mathrm{S}}\to  G$ set as
		\[
		\tilde{f}(\seq{[t\restrict i]:0<i\le |t|})=f(t)
		\]
		is a well defined $\A$-morphism.
		
It is then easy to check that $\tilde{f}$ is the unique $\A$-morphism which makes the following diagram commute for every $n<\omega$:
		\begin{center}
			\begin{tikzcd}
				\mathrm{S}n	\arrow[r, "q_n"] \arrow[dr, "f_n"]	&	 G_{\mathrm{S}} \arrow[d, "\tilde{f}"]\\
				&	G.	
			\end{tikzcd}
		\end{center}
	\end{proof}
	
	Of course, the procedure taken to compute the colimit of the sequence in Lemma \ref{LEMMA_GameASeqLim} is simply the process of taking some coequalizer after the coproduct. The benefit of explicitly showing the equivalence relation for the coequalizer is to help us show:
    
\begin{prop}
    A sequence $\mathrm{S}\colon \omega\to \mathbf{Fin}\Games_{\emb}$ has a colimit in $\Games_{\emb}$  (as a category in its own right) if, and only if, the cocone $(q_n\colon \mathrm{S}n\to G_{\mathrm{S}})_{n<\omega}$ of Lemma \ref{LEMMA_GameASeqLim} is such that for every $R\in \Run(G_{\mathrm{S}})$ there are an $n<\omega$ and an $R'\in \Run(\mathrm{S}n)$ such that $\overline{q_n}(R') = R$.
\end{prop}
\begin{proof}
    Suppose not, in which case we denote $\mathrm{S}n = (T_n,A_n)$ and then fix an $R\in  \Run(T_{\mathrm{S}})$ which is not in the image of any $\overline{q_n}$
    and a cocone $(g_n\colon \mathrm{S}n\to G)_{n<\omega}$ in $\Games_{\emb}$ with $G = (T,A)$. 
    
    Since $T_{\mathrm{S}}$ with $(q_n\colon T_n\to T_{\mathrm{S}})_{n<\omega}$ is the colimit cocone in $\Gmes$ of $\mathrm{T}$ (with $\mathrm{T}n = T_n$), the cocone of maps $(g_n\colon \mathrm{S}n\to G)_{n<\omega}$ factors through $(q_n\colon T_n\to T_{\mathrm{S}})_{n<\omega}$ and some $f\colon T_{\mathrm{S}}\to T$.

    Then, since $R=\bigcup_{n<\omega}q_{k_n}(t_{n})$ for some $\seq{k_n<\omega: n<\omega}$ and $\seq{t_n\in T_{k_n}:n<\omega}$, any chronological map $h\colon T\to T_{\mathrm{S}}$ such that 
    \[h\circ (g_n\colon T_n\to T)_{n<\omega} = (q_n\colon T_n\to T_{\mathrm{S}})_{n<\omega}\]
    must be such that $\overline{h}(\overline{f}(R)) = R$.
    
    If $\overline{f}(R)\in A$ and we consider $G' = (T_{\mathrm{S}}, A_{\mathrm{S}}\setminus\{R\})$, then $(q_n\colon  \mathrm{S}n\to G')_{n<\omega}$ cannot factor through $(g_n\colon \mathrm{S}n\to G)_{n<\omega}$ (because $\overline{h}(\overline{f}(R)) = R\notin A_{\mathrm{S}}\setminus\{R\}$ despite $\overline{f}(R)\in A$, so the desired $h$ cannot be an $\A$-morphism), thus $(g_n\colon \mathrm{S}n\to G)_{n<\omega}$ is not a colimit cocone. 
    
    On the other hand, if $\overline{f}(R)\notin A$ and we consider $G' = (T_{\mathrm{S}}, A_{\mathrm{S}}\cup\{R\})$, then any $\A$-morphism $h\colon G\to G'$ such that 
    \[h\circ (g_n\colon T_n\to T)_{n<\omega} = (q_n\colon T_n\to T_{\mathrm{S}})_{n<\omega}\]
    cannot be a $\B$-morphism  (since, in this case,  $\overline{h}(\overline{f}(R)) = R\in A_{\mathrm{S}}\cup\{R\}$ despite $\overline{f}(R)\notin A$, so $h$ cannot be an $\A$-morphism), thus $(g_n\colon \mathrm{S}n\to G)_{n<\omega}$ is not a colimit cocone in this case either.

    The proof for the other implication is straightfoward.
\end{proof}

So, since the constructed Fra\"iss\'e sequence in $\mathbf{Fin}\Games_{\emb}$ is one of many such examples in which its colimit {\em ``has branches that do not appear in the sequence itself''}, its colimit does not lie in $\Games_{\emb}$.

Nevertheless, we will show that $G_{\FL}$ is ultrahomogeneous w.r.t. $(\emb,\FinGame)$. To this end, we shall first work on some adaptations of results in the model theoretical framework:

\begin{lemma}\label{LEMMA_UH->WH}
	Suppose $G_1=(T_1,A_1)$ and $G_2=(T_2,A_2)$ are such that $T_1, T_2, A_1$ and $A_2$ are at most countable, that $G_2$ is injective in $\Games_{\emb}$ w.r.t. embeddings in $\FinGame$. If $G\le G_1$ is finite and $f\colon G \embed G_2$, then $f$ extends to an $\tilde{f}\colon G_1 \embed G_2$.
\end{lemma}
\begin{proof}
	Suppose $G = (T,A)\le G_1$ and $f\colon G \embed G_2$ ($G$ being finite). Since $G_1$ is countable, let $\seq{G_1^n = (T_1^n,A_1^n):n<\omega}$ be an \emph{exhaustion} for $G_1$, meaning that it is an increasing chain of subgames of $G_1$ with $T_1=\bigcup_{n<\omega}T_1^n$ and $A_1 = A_1^n$.
	
	By recursion, suppose we have defined an embedding $f_n\colon G_1^n\cup G  \embed G_2$ extending $f$. Since $G_1^n\cup G\le G_1^{n+1}\cup G'\le G_1$, we can use the $\FinGame_\emb$-injectivity of $G_2$ in order to find an embedding $f_{n+1}\colon G_1^{n+1}\cup G \embed G_2$ extending $f_n$. 
	
	We claim that $\tilde{f}=\bigcup_{n<\omega}f_{n}\colon T_1\to T_2$ is an embedding, which concludes the proof (because each $f_n$ extends $f$).
	
	Indeed, $\tilde{f}$ is injective, because each $f_{n}$ is injective. Moreover, because $\set{G_1^n:n<\omega}$ is an exhaustion of $G_1$, one has $R\in A_1$ if, and only if, $R$ is a run won by $\ali$ in some $G_1^n$. But $R$ is a run won by $\ali$ in $G_1^n$ if, and only if, $\overline{f_{n}}(R)\in A_2$ (because $f_n$ is an embedding). Hence, because $\tilde{f}$ extends $f_n$ for every $n<\omega$, it follows that $R\in A_1$ if, and only if, $\overline{\tilde{f}}(R)\in A_2$, from which we conclude that $\tilde{f}$ is an embedding.
\end{proof}

\begin{lemma}\label{LEMMA_WH->UH}
	Suppose $G_1=(T_1,A_1)$ and $G_2=(T_2,A_2)$ are both injective in $\Games_\emb$ w.r.t. embeddings in $\FinGame$, with $T_1,T_2,A_1$ and $A_2$ at most countable. If $G\le G_1$ is finite and $f\colon G'  \embed G_2$. Then $f$ extends to an isomorphism $\tilde{f}\colon G_1 \to G_2$.
\end{lemma}
\begin{proof}
	Let $G'\le G_1$ and $f\colon G'  \embed G_2$ ($G'$ being finite). 
	
	Let $\set{G_1^n:n<\omega}$ be an exhaustion for $G_1$ and $\set{G_2^n:n<\omega}$ be an exhaustion for $G_2$.
	
	Since $G_2$ is injective w.r.t. $\FinGame_\emb$'s morphisms, there is an embedding $G_1^0\cup f_0\colon G' \embed G_2$ extending $f$. Note that $f_0$ is an isomorphism between $G_1^0\cup G'$ and $f_0[G_1^0\cup G']$. We may then use $\FinGame_\emb$-injectivity of $G_1$ to get an embedding $g_0\colon G_2^0\cup f_0[G'\cup G_1^0] \embed G_1$ extending $f_0^{-1}$.  
	
	Recursively, suppose embeddings $f_n\colon \tilde{G}_1 \embed G_2$ and $g_n\colon \tilde{G}_2 \embed G_1$ have been defined in such a way that $\tilde{G_1}\le G_1$, $\tilde{G_2}\le G_2$, $f_{n}[\tilde{G}_1]\le \tilde{G}_2$ and $g_n$ extends $f_n^{-1}$ over $f_{n}[\tilde{G}_1]$.
	
	Then we use $\FinGame_\emb$-injectivity of $G_2$ to obtain an embedding $f_{n+1}\colon G_1^{n+1}\cup g_{n}[\tilde{G}_2] \embed G_2$ which extends $g_{n}^{-1}$ (and, therefore, also extends $f_n$) and then we use $\FinGame_\emb$-injectivity of $G_1$ again to obtain an embedding ${g_{n+1}}\colon G_2^{n+1}\cup f_{n+1}[G_1^{n+1}\cup g_{n}[\tilde{G}_2]] \embed G_1$ which extends $f_{n+1}^{-1}$ (and, therefore, also extends $g_n$).
	
	We claim that $\tilde{f}=\bigcup_{n<\omega}f_{n}\colon T_1\to T_2$ is an isomorphism, which concludes the proof (because each $f_n$ extends $f$).
	
	Indeed, $\tilde{f}$ is injective, because each $f_{n}$ is injective. Moreover, we constructed $\set{f_n:n<\omega}$ in such a way that the image of $f_{n+1}$ contains $G_2^n$, so $\tilde{f}$ is surjective (because $\set{G_2^n:n<\omega}$ is an exhaustion of $G_2$). 
	
	Finally, let $R\in \Run(T_1)$. Then, because $\set{G_1^n:n<\omega}$ is an exhaustion of $G$, $R\in A_1$ if, and only if, $R$ is a run won by $\ali$ in some $G_1^n$. But $R$ is a run won by $\ali$ in $G_1^n$ if, and only if, $\overline{f_{n}}(R)\in A_2$ (because $f_n$ is an embedding). Hence, because $\tilde{f}$ extends $f_n$ for every $n<\omega$, it follows that $R\in A_1$ if, and only if, $\overline{\tilde{f}}(R)\in A_2$, from which we conclude that $\tilde{f}$ is an isomorphism.
\end{proof}

\begin{cor}\label{COR_UH<->WH}
	\sloppy A game $G=(T,A)$ is ultrahomogeneous with respect to $(\emb,\FinGame)$ if, and only if, it is injective in $\Games_\emb$ w.r.t. embeddings in $\FinGame$.
\end{cor}
\begin{proof}
	Just apply Lemmas \ref{LEMMA_UH->WH} and \ref{LEMMA_WH->UH} with $G_1=G_2=G$.
\end{proof}

Now, going back to our candidate for \textit{Fra\"iss\'e limit}:

\begin{cor}\label{COR_GFLUlthom}
	$G_{\FL}$ is ultrahomogeneous w.r.t. $(\emb,\FinGame)$.
\end{cor}
\begin{proof}
    \sloppy We showed in Theorem \ref{THM_GFL_Injective} that $G_{\FL}$ is injective w.r.t. $\FinGame_\emb$'s morphisms and, in view of Corollary \ref{COR_UH<->WH}, this suffices to show that $G_{\FL}$ is ultrahomogeneous w.r.t. $(\emb,\FinGame)$.
\end{proof}

Summarizing the above results we obtain a result analogous to the one which is is true in the model theoretical approach:
\begin{thm}\label{THM_GamesFL}
	
		 $G_{\FL}$ is the unique (up to isomorphism) game which:
         \begin{itemize}
             \item is \emph{countable} (in the sense that its game tree and payoff set are countable),
             \item contains a copy of every finite game
             \item is ultrahomogeneous w.r.t. $(\emb,\FinGame)$.
         \end{itemize}
\end{thm}
\begin{proof}
That $G_{\FL}$ satisfies all the above follows from Proposition \ref{PROP_NoL}, Lemma \ref{LEMMA_GFL_Age} and Corollary \ref{COR_GFLUlthom}.
	
	Now, suppose $G$ satisfies the above. Then $G$ is injective w.r.t. $\FinGame_\emb$'s morphisms, so we get the desired result by applying Lemma \ref{LEMMA_WH->UH}.
\end{proof}

\section{Weakly finitely small games}\label{SEC_WeaklyFiniteSmall}

The proofs seen in the literature for the fact that the colimit in $\mathbf{C}$ of a Fra\"iss\'e sequence $\mathrm{F}\colon \omega \to \mathbf{F}_\mM$ is ultrahomogeneous with respect to $(\mM, \mathbf{F})$ relies heavily in the fact that objects in $\mathbf{F}$ are somewhat $\mM$-\textit{small} in $\mathbf{C}$. To be precise:

\begin{defn}\label{DEFN_WeaklyFinSmall}
	Suppose $\mathbf F$ is a full subcategory of $\bC$ and $\mM$ is a class of $\bC$-morphisms. We say that an object $X$ in $\mathbf{C}$ is \emph{weakly finitely small} with respect to $(\mM, \mathbf{F})$ if for every sequence $\mathrm{S}\colon \omega\to \mathbf{F}_\mM$ with colimit cocone $q_n\colon \mathrm{S}n\to \colim \mathrm{S}$ in $\mathbf{C}$ and morphism $f\colon X \to  \colim \mathrm{S}$ in $\mM$ there is an $n<\omega$ and an $\tilde{f}\colon X\to \mathrm{S}n$ in $\bF_\mM$ such that $q_n \circ \tilde{f}=f$.

    When $\mathbf{F}_\mM = \mathbf{C}_\mM = \mathbf{C}$, we simply say that $X$ is weakly finitely small in $\mathbf{C}$.
\end{defn}

We should point out that the above definition is a generalization of the concept of \emph{finite smallness} which can be seen in, e.g., \cite{Quillen1967} and, more recently, \cite{Adámek2023}, which states that an object $X$ in $\mathbf{C}$ is finitely small if the Hom functor $\mathbf{C}(X,-)\colon \mathbf{C}\to \Sets$ preserves colimits of sequences of monomorphisms. Definition \ref{DEFN_WeaklyFinSmall} is then a weakening of said property in the case which $\mathbf{F}_\mM = \mathbf{C}_\mM = \mathbf{C}$ 
 because we just ask that $\bigsqcup_{n<\omega}\mathbf{F}_\mM(X,X_n)\to\mathbf{C}_\mM(X,X_{\infty})$ is surjective.

 As it was hinted:

 \begin{thm}\label{THM_WeakFinSmallUltrahom}
     Suppose $\bC$, $\bF$ and $\mM$ are such that every object in $\mathbf{F}$ is weakly finitely small w.r.t. $(\mM, \mathbf{F})$ and $\mathrm{F}\colon \omega\to \mathbf{F}_\mM$ is a Fra\"iss\'e sequence in $\mathbf{F}_\mM$ with colimit $F$ in $\mathbf{C}$. Then $F$ is ultrahomogeneous w.r.t. $(\mM, \mathbf{F})$.
 \end{thm}
 \begin{proof}
     Fix a colimit cocone $\seq{f_n\colon \mathrm{F}n\to F}_{n\in\omega}$ in $\mathbf{C}$ and let $g,h\colon A\to F$ in $\mM$ be given with $A$ in $\mathbf{F}$. Since $A$ is weakly finitely small w.r.t. $(\mM, \mathbf{F})$, there must be $n_g,n_h\in\omega$ together with $\tilde{g}\colon A\to \mathrm{F}n_g$ and $\tilde{h}\colon A\to \mathrm{F}n_h$ in $\mathbf{F}_\mM$ such that $g = f_{n_g}$ and $h = f_{n_h}$. Without loss of generality, let us assume that $n_g=n_h=N$. 

     Consider the sequences $\mathrm{S}\colon \omega\to \mathrm{F}_\mM$ such that 
     \begin{itemize}
         \item $\mathrm{S}0 = \mathrm{S'}0 = A$ and \[\mathrm{S}n = \mathrm{S'}n = \mathrm{F}(N+n-1)\] for every $n\ge 1$.
         \item $\mathrm{S}_0^1=\tilde{g}\colon A\to \mathrm{F}N$, $\mathrm{S'}_0^1=\tilde{h}\colon A\to \mathrm{F}N$ and \[\mathrm{S}_n^{n+1} = \mathrm{S'}_n^{n+1} = \mathrm{F}_{N+n-1}^{N+n}\colon \mathrm{F}(N+n-1)\to \mathrm{F}(N+n)\] for every $n\ge 1$.
     \end{itemize}
     Then the cocones $\seq{g_n\colon \mathrm{S}n\to F}_{n\in\omega}$ and $\seq{h_n\colon \mathrm{S'}n\to F}_{n\in\omega}$ in $\mathbf{C}$ with $g_0=g\colon A\to F$, $h_0=h\colon A\to F$ and \[g_n = h_n = f_{N+n-1}\colon \mathrm{F}(N+n-1)\to F\] for every $n\ge 1$, are colimit cocones in $\mathbf{C}$ (since, from $n=1$ onward, they are equal to a final segment of the fixed colimit cocone of $\mathrm{F}$ in $\mathbf{C}$). In this case, the universal property of $\seq{g_n\colon \mathrm{S}n\to F}_{n\in\omega}$ and $\seq{h_n\colon \mathrm{S'}n\to F}_{n\in\omega}$ applied to one another provides us with the desired isomorphism $\varphi\colon F\to F$ in $\mathbf{C}$.
 \end{proof}

But, in view of Example \ref{EX_notFinRep}, not every finite game is weakly finitely small  w.r.t. $(\emb, \FinGame)$. It is then only natural to ask which are these weakly finitely small games  w.r.t. $(\emb, \FinGame)$ and what kind of Fra\"iss\'e theory emerges from them. 

Instead of starting by tackling this matter, though, we will first focus on the (harder) problem of which games are weakly finitely small in $\Games_\A$, so that the answer relating to $(\emb, \FinGame)$ will follow from it (see Theorem \ref{THM_WeakFinSmallGames_Emb}).

Our goal now is to show:

\begin{thm}\label{THM_WeakFinSmallGames}
	A game $G=(T,A)$ is weakly finitely small in $\Games_{\A}$ if, and only if, $G$ is a finite game that is trivial for $\ali$, that is,  $A=\Run(T)$.
\end{thm}

In order to do this, given a non-empty game tree $T$ and $t\in T$, let 
\begin{gather*}
    \overline{t} = \set{R\in \Run(T): R\restrict |t| = t},\\
    T(t) = \set{s\in T: s\subseteq t \text{ or } t\subseteq s}.
\end{gather*}

If $\Run(T)$ is infinite, one can recursively find a family $\set{t_n\in T:n\in\omega}$ such that $\set{\overline{t_n}:n\in\omega}$ partitions $\Run(T)$ with $\overline{t_n}\neq\emptyset$ for every $n\in\omega$. In this case:

\begin{prop}\label{PROP_InfiniteNotSmall}
    If $G = (T,A)$ is such that $\Run(T)$ is infinite, then $G$ is not weakly finitely small in $\Games_\A$.
\end{prop}
\begin{proof}
    Fix a family $\set{t_n\in T:n\in\omega}$ such that $\set{\overline{t_n}:n\in\omega}$ partitions $\Run(T)$ with $\overline{t_n}\neq\emptyset$ for every $n\in\omega$. We let $\mathrm{S}\colon \omega\to \Games_\A$ be such that 
    \begin{itemize}
        \item $\mathrm{S}n = (T_n,A_n)$ with $T_n = \bigcup_{k\le n}T(t_k)$ and $A_n = A\cap \Run(T_n)$;
        \item $\mathrm{S}_n^m\colon G_n\embed G_m$ is the inclusion map.
    \end{itemize}
    Then $\seq{i_n\colon \mathrm{S}n\embed G}_{n\in\omega}$, where $i_n\colon \mathrm{S}n\to G$ is also the inclusion map, is clearly the colimit cocone of $\mathrm{S}$ in $\Games_\A$.

    Now consider $\id_G\colon G\to G$. Then there can be no $n\in\omega$ and $f\colon G\to \mathrm{S}n$ such that $i_n\circ f = \id_G$, since $\id_G[T] = T$, but $i_n[T_n] = \bigcup_{k\le n}T(t_k)$ and $T\setminus \bigcup_{k\le n}T(t_k)\neq\emptyset$ because $\overline{t_{n+1}}\neq\emptyset$ for every $n\in\omega$.

    Hence, $G$ is not weakly finitely small in $\Games_\A$.
\end{proof}

For a non-empty game tree $T$ and $t\in T$, let

\[
    \lfloor t\rfloor = \set{s\in T: t\subsetneq s}.
\]

\begin{prop}\label{PROP_NonTrivialNotSmall}
    If $G = (T,A)$ is such that $\Run(T)\setminus A\neq\emptyset$, then $G$ is not weakly finitely small in $\Games_\A$.
\end{prop}
\begin{proof}
     Fix $R\in \Run(T)\setminus A$ and $b$ such that $b\notin \M(T)$. Then we let $\mathrm{S}\colon \omega \to \Games_\A$ be such that
     \begin{itemize}
        \item $\mathrm{S}n = (T_n,A_n)$, where
         \begin{gather*}
             T_n = T\setminus \lfloor R\restrict n\rfloor\cup\set{R\restrict k^\smallfrown \seq{b:i\le m}:k\le n, \,m\in\omega}\\
             A_n = A\cap \Run(T_n).
         \end{gather*}
         \item $\mathrm{S}_n^m\colon T_n\embed T_m$ is the inclusion map.
     \end{itemize}
    Let $i_n: \mathrm{S}n\embed (T',A))$ be the inclusion map for each $n$, where
    \[T' = T\setminus \cup\set{R\restrict k^\smallfrown \seq{b:i\le m}:k,m\in\omega}.\]
    Then the sequence of  maps $\seq{i_n}_{n\in\omega}$ is clearly the colimit cocone in $\Games_\A$ of $\mathrm{S}$. 
    
    Now consider the inclusion map $i\colon G\embed G'$. Then there can be no $n\in\omega$ with $f\colon G\to \mathrm{S}n$ such that $i_n\circ f = i$, since $R\in \overline{i}[\Run(T)]$, but $R\notin \overline{i_n}[\Run(T_n)]$ for every $n\in\omega$. Hence, $G$ is not weakly finitely small in $\Games_\A$.
\end{proof}

To finish the proof of Theorem \ref{THM_WeakFinSmallGames}, it remains to show that every finite game which is trivial for $\ali$ is weakly finitely small in $\Games_\A$. To this end, we need the following technical lemma.

	\begin{lemma}\label{LEMMA_SeqColimR1R2}
		Let  $\mathrm{S}\colon\omega \to \Games_{\A}$ be a sequence with $\mathrm{S}n=(T_n,A_n)$, colimit cocone  $\set{q_k\colon \mathrm{S}k \to G_\mathrm{S}:k\in\omega}$. Suppose $x\in A_k$ and $R\in A_\mathrm{S}$ are such that $\overline{q}_k(x)\neq R$. 
		
		Then there are $M\ge k$ and $y\in A_M$ such that $\overline{q}_M(y)=R$ and
		\[
		\Delta(y,\overline{\mathrm{S}_k^M}(x))=\Delta(R,\overline{q}_k(x)).
		\]
	\end{lemma}
	\begin{proof}
		By the definition of $A_\mathrm{S}$, there are $m\ge k$ and $R'\in A_m$ such that $\overline{q}_m(R')=R$. Let $S=\overline{\mathrm{S}_k^m}(x)$ and then $l=\Delta(R,\overline{q}_m(S))-1$. Then, by definition of $\Delta(R,\overline{q}_m(S))$, $R\restrict l = \overline{q}_m(S)\restrict l$.  
		
		In this case, because $q_m$ is chronological, 
		\[
		\Delta(R',S)\le \Delta(\overline{q}_m(R'),\overline{q}_m(S))=\Delta(R,\overline{q}_k(x)).
		\]
		
		But $q_m(R'\restrict l)=R\restrict l=\overline{q}_k(x)\restrict l=\overline{q}_m(S)$, so $R'\restrict l\sim S\restrict l$ and thus there is an $M\ge m$ such that $\mathrm{S}(m\le M)(R'\restrict l)=\mathrm{S}_m^M(S\restrict l)=\mathrm{S}_k^M(x\restrict l)$. 
		
		Finally, by letting $y=\overline{\mathrm{S}_m^M}(R')$, we get that
		\[
		\Delta(y,\overline{\mathrm{S}_k^M}(x))\ge l+1 = \Delta(R,\overline{q}_k(x)),
		\]
		so that the conclusion follows from the fact that $q_M\colon \mathrm{S}M\to  G_{\mathrm{S}}$ is chronological.
	\end{proof}

At last:

\begin{myproof}{Theorem}{THM_WeakFinSmallGames}
	Let  $\mathrm{S}\colon\omega \to \Games_{\A}$ be a sequence with $\mathrm{S}n=(T_n,A_n)$. Consider its colimit cocone $\set{q_k\colon \mathrm{S}k \to G_{\mathrm{S}}:k\in\omega}$ as in Lemma \ref{LEMMA_GameASeqLim}. 
	
	Now, in order to show that $G$ is weakly finitely small, we can assume that $G\le G_{\mathrm{S}}$ and it suffices to show that the inclusion $i\colon G  \embed G_{\mathrm{S}}$ factors through some $q_M$. So enumerate $\Run(T)=\set{R_i:i\le N}\subset A_{\mathrm{S}}$. 
	
	For each $i\le N$ we will recursively find $m_i\in \omega$ and $\set{S_j^i:j\le i}$ such that $\overline{q_{m_i}}(S_j^i)=R_j$ and $\Delta(S_{j_1}^i,S_{j_2}^i)=\Delta(R_{j_1},R_{j_2})$. For $i=0$ we may simply use the definition of $A_{\mathrm{S}}$ to find an $S_0^0$ in some $A_{m_0}$ with $\overline{q_{m_0}}(S_0^0)=R_0$.
	
	Suppose the work is done for $i< N$. Then we may use Lemma \ref{LEMMA_SeqColimR1R2} to find $k_0\ge m_i$ and $y_0\in A_{k_0}$ such that $\overline{q}_{k_0}(y_0)=R_{i+1}$ and
	\[
	\Delta(y_0,\overline{\mathrm{S}(m_i\le k_0)}(R_0^i))=\Delta(R_{i+1},R_0).
	\]
	
	We then may use Lemma \ref{LEMMA_SeqColimR1R2} again to find $k_1\ge k_0$ and $y_1\in A_{k_1}$ such that $\overline{q}_{k_1}(y_1)=R_{i+1}$ and
	\[
	\Delta(y_1,\overline{\mathrm{S}(k_0\le k_1)}(R_1^i))=\Delta(R_{i+1},R_1).
	\]
	
	Note that
	\[
	\Delta(y_1,\overline{\mathrm{S}(m_i\le k_1)}(R_0^i))=\Delta(R_{i+1},R_0).
	\]
	
	By repeating this process $i$ times we then get the desired $m_i\in \omega$ and $\set{S_j^i:j\le i}$.
	
	Let $M=m_N$. It is then easy to check that the mapping $\overline{f}\colon \Run(T)\to A_{\mathrm{S}}$ such that
	\[
	\overline{f}(R_i)=S_i^N
	\]
	defines an $\A$-morphism $f\colon G\to \mathrm{S}M$ such that $q_M  f$ is the identity over $G$, which concludes the proof.
	
 The proof is then concluded with Propositions \ref{PROP_InfiniteNotSmall} and \ref{PROP_NonTrivialNotSmall}.
\end{myproof}

As previously mentioned, we are particularly interested in determining which are the weakly finitely small games w.r.t. $(\emb, \FinGame)$. However, note that we, purposefully, only used game embeddings in the proof of Proposition \ref{PROP_NonTrivialNotSmall}, so the same proof works to show that
\begin{prop}
    If $G = (T,A)$ is such that $\Run(T)\setminus A\neq\emptyset$, then $G$ is not weakly finitely small w.r.t. $(\emb, \FinGame)$.
\end{prop}
Finally:

\begin{thm}\label{THM_WeakFinSmallGames_Emb}
    A game $G=(T,A)$ is weakly finitely small in $\Games_{\emb}$ w.r.t. $(\emb, \FinGame)$ if, and only if, $G$ is a finite game that is trivial for $\ali$.
\end{thm}
\begin{proof}
    It remains to show that finite games which are trivial for $\ali$ are weakly finitely small w.r.t. $(\emb, \FinGame)$.
    
    So let $G=(T,A)$ be a finite game which is trivial for $\ali$ and suppose $\mathrm{S}\colon \omega\to \mathbf{Fin}\Games_{\emb}$ is given with colimit cocone $\seq{q_n\colon \mathrm{S}n\to G_{\mathrm{S}}}$ in $\Games_\A$ and an embedding $f\colon G\embed G_\mathrm{S}$.

    By Theorem \ref{THM_WeakFinSmallGames}, $G$ is weakly finitely small in $\Games_\A$, so there must be an $n\in\omega$ and $g\colon G\to\mathrm{S}n$ in $\Games_\A$ such that $q_n\circ g = f$. In this case, since $f$ is a game embedding, so must $g$. Thus, $g$ attests that $G$ is weakly finitely small w.r.t. $(\emb, \FinGame)$ for $\mathrm{S}$, $\seq{q_n\colon \mathrm{S}n\to G_{\mathrm{S}}}$ and $f\colon G\embed G_\mathrm{S}$.
\end{proof}

Consider the functor $\mathrm{Free}\colon \Gmes\to \Games_\A$ such that $\mathrm{Free}T = (T,\Run(T))$ and \[\mathrm{Free}\left(T\stackrel{f}{\to} T'\right) = (T,\Run(T))\stackrel{f}{\to} (T',\Run(T')).\]

Then $\mathrm{Free}$ is clearly an isomorphism over its image in $\Games_\A$, which comprises of the full subcategory of $\Games_\A$ whose objects are games that are trivial for $\ali$. Thus, if $\mathbf{Fin}\Gmes$ is the full subcategory of $\Gmes$ comprised of trees with finitely many runs, then $\mathrm{Free}[\mathbf{Fin}\Gmes]$ is \emph{exactly} the full subcategory of $\Games_\A$ comprised of finite games which are trivial for $\ali$, i.e. (in view of Theorem \ref{THM_WeakFinSmallGames}), the full subcategory of $\Games_\A$ comprised of those games which are weakly finitely small in $\Games_\A$.

Thus, even though $\mathrm{Free}\colon \mathbf{Fin}\Gmes\to \mathrm{Free}[\mathbf{Fin}\Gmes]$ is an isomorphism, the following result highlights the crucial role of the ambient category $\mathbf{C}$ in Definition \ref{DEFN_WeaklyFinSmall}:

\begin{prop}\label{PROP_WeakFinSmallGmes}
	A decision tree $T$ is weakly finitely small in $\Gmes$ if, and only if, $T=\emptyset$.
\end{prop}
\begin{proof}
	As $T=\emptyset$ is the initial object in $\Gmes$, it is clear that it is weakly finitely small in $\Gmes$.
	
	The rest of the proof will be pretty much the same as the proof of Proposition \ref{PROP_NonTrivialNotSmall}:
 
    Fix $R\in \Run(T)$ and $b$ such that $b\notin \M(T)$. Then we let $\mathrm{S}\colon \omega \to \Gmes$ be such that
     \begin{itemize}
        \item $\mathrm{S}n = T_n$, where
         \begin{gather*}
             T_n = T\setminus \lfloor R\restrict n\rfloor\cup\set{R\restrict k^\smallfrown \seq{b:i\le m}:k\le n, \,m\in\omega}.
         \end{gather*}
         \item $\mathrm{S}_n^m\colon T_n\embed T_m$ is the inclusion map.
     \end{itemize}
     Then $\seq{i_n\embed \mathrm{S}n\to T')}_{n\in\omega}$ with 
    \[T' = T\setminus \cup\set{R\restrict k^\smallfrown \seq{b:i\le m}:k,m\in\omega}\]    and each $i_n$ being the inclusion map is clearly the colimit cocone in $\Gmes$ of $\mathrm{S}$. 
    
    Now consider the inclusion map $i\colon T\embed T'$. Then there can be no $n\in\omega$ with $f\colon G\to \mathrm{S}n$ such that $i_n\circ f = i$, since $R\in \overline{i}[\Run(T)]$, but $R\notin \overline{i_n}[\Run(T_n)]$ for every $n\in\omega$. Hence, $T$ is not weakly finitely small in $\Gmes$.
\end{proof} 

At last:

\begin{thm}\label{THM_AFraisseGame}
	There is a Fra\"iss\'e sequence $\mathrm{F}_{\A}\colon \omega \to \mathrm{Free}[\mathbf{Fin}\Gmes]_{\emb}$. Moreover, the colimit in $\Games_\A$ of $\mathrm{F}_{\A}$ is $G_{\FL}$.
\end{thm}
\begin{proof}
	Let $G_n=(T_n,A_n)$ be as in Example \ref{EX_notFinRep} and then define $\mathrm{F}_{\A}\colon \omega \to \mathrm{Free}[\mathbf{Fin}\Gmes]_{\emb}$ as:
	\begin{itemize}
		\item On objects, $\mathrm{F}_{\A}n=G_n$.
		\item On morphisms, let ${\mathrm{F}_{\A}}_n^m$ be the inclusion of $G_n$ into $G_m$ (recall that $G_n\le G_m$).
	\end{itemize}
	Then it is clear that $\seq{i_n\colon G_n\embed G_{\FL}}$ is the colimit cocone in $\Games_{\A}$ of $\mathrm{F}_{\A}$, where each $i_n\colon G_n\embed G_{\FL}$ is the inclusion map. 
	
	Now let us show that $\mathrm{F}_{\A}$ is a Fra\"iss\'e sequence. 
	
	In order to show condition (a) of Definition \ref{DEF_FraisseSeq}, let $G$ be a finite game that is trivial for $\ali$. Then, because $G$ is finite, it follows from Lemma \ref{LEMMA_GFL_Age} that there is an embedding $f\colon G \embed G_{\FL}$. So condition (a) follows from the fact that $G$ is weakly finitely small in $\Games_\A$ and that $G_{\FL}$ is the colimit in $\Games_\A$ of $\mathrm{F}_{\A}$ (we use here once again the fact that if $f=g\circ h$ is an embedding, then $h$ must be an embedding).
	
	Finally, striving to show that condition (b) of Definition \ref{DEF_FraisseSeq} holds, suppose that $G$ is a finite game that is trivial for $\ali$ and $f\colon G_n  \embed G$ is an embedding. Then, because $G_{\FL}$ is injective in $\Games_{\emb}$ w.r.t. $\FinGame_{\emb}$'s morphisms, there exists $\tilde{f}\colon G  \embed G_{\FL}$ extending $i_n$ through $f$. In this case, it follows again from the fact that $G$ is weakly finitely small in $\Games_\A$ that there is an $m\ge n$ and $g\colon G \embed G_m$ such that $\tilde{f} = i_m\circ g$. But $i_n = i_m\circ {\mathrm{F}_\A}_n^m$ and
     \[
        i_n = \tilde{f}\circ f = i_m\circ g\circ f,
     \] 
     thus monicity of $i_m$ tells us that ${\mathrm{F}_\A}_n^m = g\circ f$, which concludes the proof.
\end{proof}

\section{Ultrahomogeneity revisited through the lens of category theory}\label{SEC_Directedness}
In view of Example \ref{EX_notFinRep} and Theorems \ref{THM_WeakFinSmallUltrahom} and \ref{THM_WeakFinSmallGames_Emb}, it is surprising that the Fra\"iss\'e limit of $\mathbf{Fin}\Games_{\emb}$ in $\Games_\A$ is ultrahomogeneous w.r.t. $(\emb, \FinGame)$. We pinpoint in this section general (and strictly weaker) categorical assumptions which hold in our setting and also guarantee ultrahomogeneity of the limit of Fra\"iss\'e sequences.

In what follows, given a pair $\mathbf{F}\subseteq\mathbf{C}$ with an object $X$ in $\bC$, we denote by $(\mathbf F/X)_\mathbf C$ the comma category $(\mathrm{I}/\mathrm{F})$ with $\mathrm{I}\colon \mathbf F\to \mathbf C$ being the inclusion and $\mathrm{F}\colon \mathbf 1 \to \mathbf C$ being such that $\mathrm{F}\ast = X$ (where $\mathbf{1}$ denotes the trivial category with object $*$). In other words, $(\mathbf F/X)_\mathbf C$ denotes the category in which:
\begin{itemize}
    \item objects are morphisms $x\colon A\to X$ in $\mathbf C$ with $A$ in $\mathbf F$
    \item a morphism from $x_A\colon A\to X$ to $x_B\colon B\to X$ is a morphism $f\colon A\to B$ in $\mathbf F$ such that $x_B\circ f=x_A$.
\end{itemize}

\begin{prop}\label{PROP_WeaffinsmallDirectedAmalg}
    Let $\mathbf{F}$ be a subcategory of $\mathbf{C}$ and suppose every object in $\mathbf{F}$ is weakly finitely small in $\mathbf{C}$. If $X$ is any object in $\mathbf{C}$ which is the colimit of some sequence in $\mathbf{F}$, then the comma category $(\mathbf{F}/X)_\mathbf{C}$ has the JEP.
\end{prop}

\begin{proof}
    Suppose $\mathrm S\colon \omega\to \mathbf F$ is a sequence with colimit cocone $(x_n\colon \mathrm{S}n\to X)_{n<\omega}$ and let $f\colon A\to X$ and $g\colon B\to X$ be given with objects $A,B$ in $\mathbf F$. Since $A$ and $B$ are weakly finitely small, there are $n_A,n_B<\omega$, $f'\colon A\to S_{n_A}$ and $g'\colon B\to S_{n_B}$ such that 
    \begin{align*}
    x_{n_A}\circ f' &= f\\   
    x_{n_B}\circ g' &= g.
    \end{align*}
    Without loss of generality, we may assume that $n=n_A=n_B$. Hence, $f'\colon A\to \mathrm{S}n$, $g'\colon B\to \mathrm{S}n$ and $x_n\colon \mathrm{S}n\to X$ attest directedness of $(\mathbf F/ X)_{\mathbf C}$ for $f\colon A\to X$ and $g\colon B\to X$.
\end{proof}

We are particularly interested in fixing a class $\mM$ of monomorphisms in $\bC$, since:
\begin{prop}\label{PROP_DirMonAmalg}
    Suppose $\mM$ is a class of monomorphisms in $\bC$ and $(\mathbf F_\mM/X)_{\mathbf{C}_\mM}$ has the JEP. Then $(\mathbf F_\mM/X)_{\mathbf{C}_\mM}$ has the AP.
\end{prop}
\begin{proof}
    Suppose $c_A\colon C\to A$, $c_B\colon C\to B$, $f\colon A\to X$ and $g\colon B\to X$ are morphisms in $\mM$ such that $c = f\circ c_A = g\circ c_B\colon C\to X$, with objects $A,B,C$ in $\mathbf{F}$. Let $f'\colon A\to D$, $g'\colon B\to D$ and $x\colon D\to X$ be the morphisms attesting the JEP of $(\mathbf{F}_\mM/X)_{\mathbf{C}_\mM}$ for $f$ and $g$. Then $f'\colon A\to D$, $g'\colon B\to D$ and $x\colon D\to X$ also attest the AP of $(\mathbf{F}_\mM/ X)_{\mathbf{C}_\mM}$ for $c_A\colon C\to A$, $c_B\colon C\to B$, $f\colon A\to X$ and $g\colon B\to X$, since monicity of $x$ guarantees that 
    \[x_n\circ f'\circ c_A = x_n\circ g'\circ c_B = c\]
    entails $f'\circ c_A = g'\circ c_B$.
\end{proof}

In our game categories:
\begin{prop}
    For every game $G$, $(\mathbf{Fin}\Games_{\emb}/G)_{\Games_\A}$ and $(\mathbf{Fin}\Games_{\emb}/G)_{\Games_{\emb}}$ has the JEP. Furthermore, $(\mathbf{Fin}\Games_{\emb}/G)_{\Games_{\emb}}$ has the AP.
\end{prop}
\begin{proof}
    We give a proof for the JEP in $(\mathbf{Fin}\Games_{\emb}/G)_{\Games_{\emb}}$, as the proof for the JEP in $(\mathbf{Fin}\Games_{\emb}/G)_{\Games_\A}$ is analogous. Suppose embeddings $f\colon G_1\embed G$ and $g\colon G_2\embed G$ are given, with $G_1 = (T_1,A_1)$, $G_2 = (T_2,A_2)$ and $G = (T,A)$. Let $G' = (f[T_1]\cup g[T_2], A\cap \Run(f[T_1]\cup g[T_2]))\le G$ and $i\colon G'\embed G$ be its inclusion map. Letting $f'$ and $g'$ be the morphisms $f$ and $g$ obtained by restricting the codomain to $G'$, we then can conclude that $f'$ $g'$ and $i\colon G'\embed G$ attest the JEP in $(\mathbf{Fin}\Games_{\emb}/G)_{\Games_{\emb}}$ for $f$ and $g$.

    The amalgamation property in $(\mathbf{Fin}\Games_{\emb}/G)_{\Games_{\emb}}$ follows from the fact that game embeddings are mono in $\Games_{\A}$, combined with Proposition \ref{PROP_DirMonAmalg}.
\end{proof}

The following concepts will also play an important role in guaranteeing ultrahomogeneity of the Fra\"iss\'e limit.

\begin{defn}\label{DEFN_Squeeze}
    Suppose $\mathbf F$ is a full subcategory of $\bC$ and $\mM$ is a class of $\bC$-morphisms.

     We say that the pair $(\mM, \bF)$ has the {\em tight squeeze property} in $\bC$ if, for all $\mathrm{S},\mathrm{S}'\colon \omega \to \mathbf{F}_\mM$ with cocone $\gamma = (\gamma_n\colon \mathrm{S}n\to X)_{n<\omega}$ and natural transformation $\theta\colon \mathrm{S}'\to \mathrm{S}$ in $\mM$:
     
     \begin{itemize}
         \item[(i)]  if $\gamma\circ\theta$ is a colimit cocone in $\bC$, then so is $\gamma$;
         \item[(ii)] if $\mathrm{S}'$ is a Fra\"iss\'e sequence in $\mathbf{F}_\mM$, then so is $\mathrm{S}$.
     \end{itemize}
\end{defn}

 We will later investigate the concepts in Definition \ref{DEFN_Squeeze} with the pair $(\emb, \FinGame)$ in $\Games_\A$, but we first motivate it with 

\begin{thm}\label{THM_SqueezeDirUltrahom}
    Suppose the pair $(\mM, \bF)$ has the tight squeeze property in $\bC$, morphisms in $\mM$ are all mono, $(f_n\colon \mathrm{F}n\to F_{\lim})_{n<\omega}$ in $\mM$ is a colimit cocone in $\mathbf{C}$ of the Fra\"iss\'e sequence $\mathrm{F}$ in $\mathbf{F}_\mM$ and $(\mathbf{F}_\mM/F)_{\mathbf{C}_\mM}$ has the JEP. Then $F_{\lim}$ is ultrahomogeneous w.r.t. $(\mM,\mathbf{F})$. 
\end{thm}

In order to prove the above result, we need the following technical lemmas.

\begin{lemma}\label{LEMMA_AmalgSqueeze}
    Suppose $\mathbf F$ is a full subcategory of $\bC$ and $\mM$ is a class of $\bC$-morphisms, let $\mathrm{S}\colon \omega\to \mathbf{F}_\mM$ be a sequence with $(x_n\colon \mathrm{S}n\to X)_{n<\omega}$ in $\mM$ being a colimit cocone in $\mathbf{C}$ and suppose that the morphisms in $\mM$ are all mono and $(\mathbf{F}_\mM/X)_{\mathbf{C}_\mM}$ has the JEP. Then for every $f\colon A\to X$ there is a sequence $\tilde{\mathrm{S}}\colon \omega\to \mathbf{F}_\mM$, a cocone $(x_n'\colon \tilde{\mathrm{S}}n\to X)_{n<\omega}$ in $\mM$, a natural transformation $\theta\colon\mathrm{S}\to\tilde{\mathrm{S}}$ and $\tilde{f}\colon A\to \tilde{\mathrm{S}}0$ in $\mathbf{F}_\mM$ such that $\tilde{x}_0\circ \tilde{f} = f$ and
    \[(x_n\colon \mathrm{S}n\to X)_{n<\omega} = (x_n'\colon \tilde{\mathrm{S}}n\to X)_{n<\omega}\circ \theta.\]
\end{lemma}

\begin{proof}
    We define the desired sequence $\tilde{\mathrm{S}}\colon \omega\to \mathbf{F}_\mM$, cocone $(x_n'\colon \tilde{\mathrm{S}}n\to X)_{n<\omega}$ and natural transformation $\theta\colon \mathrm{S}\to \tilde{\mathrm{S}}$ recursively and simultaneously. We start by applying directedness of $(\mathbf C/X)_\mathbf D$ to $f\colon A\to X$ and $x_0\colon \mathrm{S}0\to X$ in order to find $\tilde{x}_0\tilde{\mathrm{S}}_0\to X$, $\tilde{f}\colon A\to \tilde{\mathrm{S}}0$ and $\theta_0\colon \mathrm{S}0\to X$ such that $\tilde{x}_0\circ\tilde{f} = f$ and $\tilde{x}_0\circ \theta_0 = x_0$ (as illustrated below).

    $$
    
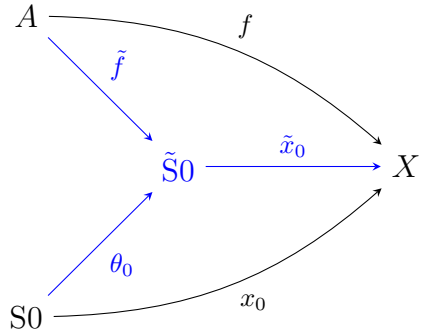
\begin{tikzpicture}[>=stealth]
    
    \node (X) at (5, 0) {$X$};
    \node (A) at (0,2) {$A$};
    \node (tS0) at (2,0) {\color{blue}$\tilde{\mathrm{S}}0$};
    \node (S0) at (0,-2) {$\mathrm{S}0$};

    {\footnotesize 
    \draw[->, color=blue] (S0) -- node[below right] {$\theta_{0}$} (tS0);
    \draw[->, color=blue] (A) -- node[above right] {$\tilde{f}$} (tS0);
    \draw[->] (A) edge[bend left=20] node[above right] {$f$} (X);

    \draw[->, color=blue] (tS0) -- node[above] {$\tilde{x}_{0}$} (X);
    \draw[->] (S0) edge[bend right=20] node[below right] {$x_0$} (X);}
    \end{tikzpicture}
    $$

    Suppose $\tilde{\mathrm{S}}$ is defined up to $k<\omega$, as well as $\theta$ and $(\tilde{x}_n\colon \tilde{\mathrm{S}}n\to X)_{n\le k}$. In view of Proposition \ref{PROP_DirMonAmalg}, we may apply the amalgamation property in $(\mathbf C/X)_\mathbf D$ to $\theta_k\colon \mathrm{S}k\to \tilde{\mathrm{S}}k$, $\mathrm{S}_k^{k+1}\colon \mathrm{S}k\to \mathrm{S}(k+1)$, $x_{k+1}\colon \mathrm{S}(k+1)\to X$ and $\tilde{x}_k\colon \tilde{\mathrm{S}}k\to X$ in order to find $\theta_{k+1}\colon \mathrm{S}(k+1)\to \tilde{\mathrm{S}}(k+1)$, $\tilde{\mathrm{S}}_k^{k+1}\colon \tilde{\mathrm{S}}k\to \tilde{\mathrm{S}}(k+1)$, and $\tilde{x}_{k+1}\colon \tilde{\mathrm{S}}(k+1)\to X$ such that $\tilde{\mathrm{S}}_k^{k+1}\circ\theta_{k} = \theta_{k+1}\circ \mathrm{S}_k^{k+1}$, $x_{k+1} = \tilde{x}_{k+1}\circ \theta_{k+1}$ and $\tilde{x}_k = \tilde{x}_{k+1}\circ \tilde{\mathrm{S}}_k^{k+1}$ (as illustrated below).
    $$
    
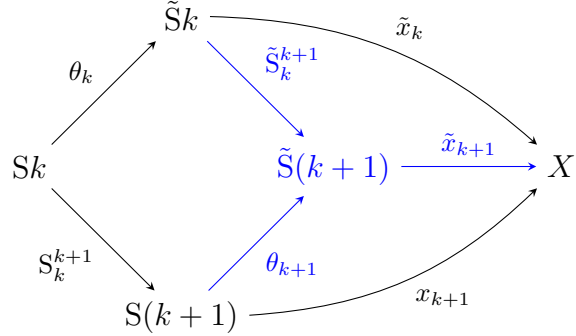
\begin{tikzpicture}[>=stealth]
    
    \node (X) at (7, 0) {$X$};
    \node (Sk) at (0,0) {$\mathrm{S}k$};
    \node (tSk) at (2,2) {$\tilde{\mathrm{S}}k$};
    \node (tSk+1) at (4,0) {\color{blue}$\tilde{\mathrm{S}}(k+1)$};
    \node (Sk+1) at (2,-2) {$\mathrm{S}(k+1)$};

    {\footnotesize 
    \draw[->] (Sk) -- node[above left] {$\theta_k$} (tSk);
    \draw[->] (Sk) -- node[below left] {$\mathrm{S}_k^{k+1}$} (Sk+1);
    \draw[->, color=blue] (Sk+1) -- node[below right] {$\theta_{k+1}$} (tSk+1);
    \draw[->, color=blue] (tSk) -- node[above right] {$\tilde{\mathrm{S}}_k^{k+1}$} (tSk+1);
    \draw[->] (tSk) edge[bend left=20] node[above right] {$\tilde{x}_k$} (X);

    \draw[->, color=blue] (tSk+1) -- node[above] {$\tilde{x}_{k+1}$} (X);
    \draw[->] (Sk+1) edge[bend right=20] node[below right] {$x_{k+1}$} (X);}
    \end{tikzpicture}
    $$
    Hence the construction is complete.
\end{proof}

Now suppose $\mathrm{S},\mathrm{R}\colon \omega\to \mathbf{F}_\mM$ are two Fra\"iss\'e sequences in $\mathbf{F}_\mM$ and $f\colon \mathrm{R}0\to \mathrm{S}n_0$ is in $\mathbf{F}_\mM$. Then, using condition (b) of \ref{DEF_FraisseSeq} for $\mathrm{R}$, there must be an $n_1> n_0$ and an $f_1\colon \mathrm{S}n_0\to \mathrm{R}n_1$ such that  $f_1\circ f = \mathrm{R}_0^{n_1}$. We can now apply condition (b) of \ref{DEF_FraisseSeq} for $\mathrm{S}$, thus obtaining $n_2> n_1$ and an $f_2\colon \mathrm{R}n_1\to \mathrm{S}n_2$ such that  $f_2\circ f_1 = \mathrm{S}_{n_0}^{n_2}$. By proceeding in this manner we construct a so called \emph{back-and-forth} between $\mathrm{R}$ and $\mathrm{S}$ which starts with the given $f\colon \mathrm{R}0\to \mathrm{S}n_0$. Formally:

\begin{defn}
    Given an order embedding $\varphi\colon \omega\to \omega$, we denote by $\varphi_{\even},\varphi_{\odd},\varphi_{+1}\colon \omega\to \omega$ the order embeddings such that
    \begin{align*}
        \varphi_{\even}(n) &= \varphi(2n),\\
        \varphi_{\odd}(n) &= \varphi(2n+1),\\
        \varphi_{+1}(n) &= \varphi(n+1).
    \end{align*}

    We say that $\mathrm{S}'\colon \omega\to \mathbf{C}$ is a subsequence of $\mathrm{S}\colon \omega\to \mathbf{C}$ if there is an order embedding $\varphi\colon \omega\to \omega$ such that $\mathrm{S}' = \mathrm{S}\circ \varphi$. In this case, we write $\mathrm{S}'k = \mathrm{S}n_k$, where $n_k=\varphi(k)$.

    A \emph{back-and-forth} between two sequences $\mathrm{R},\mathrm{S}\colon \omega\to \mathbf{C}$ in a category $\mathbf{C}$ is an order embedding $\varphi\colon\omega \to \omega$ together with a pair of natural transformations 
    \begin{align*}
        \theta\colon \mathrm{R}\circ \varphi_{\even}\to \mathrm{S}\circ \varphi_{\odd}\\
        \eta\colon \mathrm{S}\circ \varphi_{\odd}\to \mathrm{R}\circ ({\varphi_{\even}})_{+1}
    \end{align*}
    such that the following diagrams commute for each $k<\omega$: 
            $$
    \begin{tikzpicture}[>=stealth]
    \node(R) at (-1,1) {$\mathrm{R}\colon$};
    \node (R0dots) at (0,1) {$\cdots$};
    \node (R0) at (2,1) {$\mathrm{R}\varphi(2k)$};
    \node (R2) at (6,1) {$\mathrm{R}\varphi(2k+2)$};
    \node (Rdots) at (8,1) {$\cdots$};
    
    \node(S) at (-1,-1) {$\mathrm{S}\colon$};
    \node (S0dots) at (2,-1) {$\cdots$};
    \node (S0) at (4,-1) {$\mathrm{S}\varphi(2k+1)$};
    \node (S1) at (8,-1) {$\mathrm{S}\varphi(2k+3)$};
    \node (S2) at (10,-1) {$\cdots$};

    {\footnotesize 
    \draw[->] (R0dots) -- node[above]{} (R0);
    \draw[->] (R0) -- node[above]{$\mathrm{R}_{\varphi(2k)}^{\varphi(2k+2)}$} (R2);
    \draw[->] (R2) --  (Rdots);

    \draw[->] (S0dots) -- node[below]{} (S0);
    \draw[->] (S0) -- node[below]{$\mathrm{S}_{\varphi(2k+1)}^{\varphi(2k+3)}$} (S1);
    \draw[->] (S1) -- node[below]{} (S2);

    \draw[->] (R0) -- node[below,left]{$\theta_k$} (S0);
    \draw[->] (S0) -- node[below,left]{$\eta_{k}$} (R2);
    \draw[->] (R2) -- node[below,left]{$\theta_{k+1}$} (S1);
    }
    \end{tikzpicture}
    $$
\end{defn} 

This concept is particularly relevant for us because:

\begin{lemma}\label{LEMMA_BackAndForth_Iso}
    Suppose $\mathbf F$ is a full subcategory of $\bC$ and $\mM$ is a class of $\bC$-morphisms, let $\mathrm{R}, \mathrm{S}\colon \omega\to \mathbf{F}_\mM$ be sequences with $(r_n\colon \mathrm{R}n\to X)_{n<\omega}$ and $(s_n\colon \mathrm{S}n\to X)_{n<\omega}$ in $\mM$ being colimit cocones in $\mathbf{C}$.

    If there is a back-and-forth between $\mathrm{R}$ and $\mathrm{S}$, then there is an isomorphism $f\colon X\to X$ which commutes with the cocones $(r_n\colon \mathrm{R}n\to X)_{n<\omega}$ and $(s_n\colon \mathrm{S}n\to X)_{n<\omega}$, together with the back-and-forth.
\end{lemma}
\begin{proof}
    Let $\varphi\colon \omega\to \omega$ be the order embedding which, together with the natural transformations $\theta\colon \mathrm{R}\circ \varphi_{\even}\to \mathrm{S}\circ \varphi_{\odd}$ and $
        \eta\colon \mathrm{S}\circ \varphi_{\odd}\to \mathrm{R}\circ {\varphi_{\even}}_{+1}$, attests the back-and-forth between $\mathrm{R}$ and $\mathrm{S}$. Note that $(r_{\varphi(2n)}\colon \mathrm{R}\varphi(2n)\to X)_{n<\omega}$ and $(s_{\varphi(2n+1)}\colon \mathrm{S}\varphi(2n+1)\to X)_{n<\omega}$ are colimit cocones as well (see, e.g., Theorem 1 in page 217 of \cite{MacLane1978}).
        
        For each $n\in\omega$, let $r_{\varphi(2n)}' = s_{\varphi(2n+1)}\circ \theta_n\colon \mathrm{R}\varphi(2n)\to X$. Then, since $\theta$ is a natural transformation, $(r_{\varphi(2n)}'\colon \mathrm{R}n\to X)_{n<\omega}$ is a cocone. In this case, since $(r_{\varphi(2n)}\colon \mathrm{R}n\to X)_{n<\omega}$ is a colimit cocone, there must be a unique $f\colon X\to X$ such that $f\circ r_{\varphi(2n)} = r_{\varphi(2n)}' = s_{\varphi(2n+1)}\circ \theta_n$ for every $n<\omega$. Similarly, we can find a unique $g\colon X\to X$ such that $g\circ s_{\varphi(2n+1)} = s_{\varphi(2n+1)}' = r_{\varphi(2n+2)}\circ \eta_n$.

    We claim that $f\circ g = g\circ f = \id_X$, which concludes the proof. Indeed, let $n<\omega$. Then 
    $$(g\circ f)\circ r_{\varphi(2n)} = g\circ r_{\varphi(2n)}' = g\circ (s_{\varphi(2n+1)}\circ \theta_n).$$ 
    On the other hand, 
    $$g\circ s_{\varphi(2n+1)} =  s_{\varphi(2n+1)}' = r_{\varphi(2n+2)}\circ \eta_n$$
    Thus, since the back-and-forth condition tells us that $\eta_n\circ \theta_n = \mathrm{R}_{\varphi(2n)}^{\varphi(2n+2)}$, we obtain that
    $$(g\circ f)\circ r_{\varphi(2n)} = r_{\varphi(2n+2))}\circ \mathrm{R}_{\varphi(2n))}^{\psi(\varphi(n))} = r_{\varphi(2n)}.$$
    In this case, it follows from the universal property of the colimit cocone that $g\circ f = \id_X$. The proof of $f\circ g = \id_X$ is analogous.
\end{proof}

At last:

\begin{myproof}{Theorem}{THM_SqueezeDirUltrahom}
    Let $g,h\colon A\to F_{\lim}$ be given. Then we use Lemma \ref{LEMMA_AmalgSqueeze} to find another sequence $\mathrm{S}\colon \omega\to \mathbf{C}$ with a cocone $(s_n\colon \mathrm{S}n\to F_{\lim})_{n<\omega}$ in $\mM$ squeezed between $\mathrm{F}$ and $F$, together with $\tilde{g}\colon A\to \mathrm{S}0$ in $\mathbf{F}_\mM$ such that $g= s_0\circ \tilde{g}$. The tight squeeze property of $(\mM,\mathbf{F})$ in $\bC$ then entails that $(s_n\colon \mathrm{S}n\to F_{\lim})_{n<\omega}$ is a colimit cocone in $\mathbf{C}$ and that $\mathrm{S}$ is a Fra\"iss\'e sequence in $\mathbf{F}_\mM$ as well.

    We now apply Lemma \ref{LEMMA_AmalgSqueeze} again, this time to $\mathrm{S}$, $(s_n\colon \mathrm{S}n\to F_{\lim})_{n<\omega}$ and $g\colon A\to F_{\lim}$ in order to find yet another Fra\"iss\'e sequence $\mathrm{R}\colon \omega\to \mathbf{F}_\mM$ with a cocone $(\tilde{r}_n\colon \mathrm{R}n\to F_{\lim})_{n<\omega}$, together with $\tilde{h}\colon A\to \mathrm{R}0$ such that $h = r_0\circ \tilde{h}$, which is squeezed between $\mathrm{S}$ and $F_{\lim}$.

    Consider the sequence $\tilde{\mathrm{S}}\colon \omega\to \mathbf{F}_\mM$ which is equal to $\mathrm{S}$ concatenated with $\tilde{g}\colon A\to \mathrm{S}_0$ so that $\tilde{\mathrm{S}}n+1 = \mathrm{S}n$ for every $n<\omega$ and $\tilde{\mathrm{S}}$ is a Fra\"iss\'e sequence (due to the amalgamation property in $\bF_\mM$). Since $\mathrm{R}$ is a Fra\"iss\'e sequence as well, there is a back-and-forth between $\tilde{\mathrm{S}}$ and $\mathrm{R}$ given by an order embedding $\varphi\colon \omega\to \omega$ together with a pair of natural transformations 
    \begin{align*}
        \theta\colon \tilde{\mathrm{S}}\circ \varphi_{\even}\to \mathrm{R}\circ \varphi_{\odd}\\
        \eta\colon \mathrm{R}\circ \varphi_{\odd}\to \tilde{\mathrm{S}}\circ {\varphi_{\even}}_{+1}
    \end{align*}
    which starts with $\theta_0 = \mathrm{R}_0^1\circ\tilde{h}$ (in particular, $\varphi(0) = 0$ and $\varphi(1) = 1$), as illustrated in the commuting diagram below.
    $$
    \begin{tikzpicture}[>=stealth]
    
    \node (Stilde) at (5.5,2) {$\tilde{\mathrm{S}}\varphi(2)$};
    \node (Sequal) at (5.5,1.5) {\rotatebox{90}{$=$}};
    \node(S) at (-1,1) {$\mathrm{S}\colon$};
    \node (A) at (0,1) {$A$};
    \node (Stilde0) at (0,2) {$\tilde{\mathrm{S}}0$};
    \node (Sphi) at (1.25,2) {$\tilde{\mathrm{S}}\varphi(0)$};
    \node (Aequal) at (0,1.5) {\rotatebox{90}{$=$}};
    \node (Sphiequal) at (0.5,2) {$=$};
    \node (S0) at (2,1) {$\mathrm{S}0$};
    \node (S1) at (5.5,1) {$\mathrm{S}(\varphi(2)-1)$};
    \node (Sdots) at (10,1) {$\cdots$};

    \node(R) at (-1,-1) {$\mathrm{R}\colon$};
    \node (R0) at (0,-1) {$\mathrm{R}0$};
    \node (R1) at (2,-1) {$\mathrm{R}1$};
    \node (R3) at (8,-1) {$\mathrm{R}\varphi(3)$};
    \node (Rdots) at (10,-1) {$\cdots$};
    \node (Rphi1) at (2,-2) {$\mathrm{R}\varphi(1)$};
    \node (Requal) at (2,-1.5) {\rotatebox{90}{$=$}};
    
    {\footnotesize 
    \draw[->] (R0) -- node[below]{$\mathrm{R}_0^1$} (R1);
    \draw[->] (R1) -- node[below]{$\mathrm{R}_{\varphi(1)}^{\varphi(3)}$} (R3);
    \draw[->] (R3) --  (Rdots);

    \draw[->] (A) -- node[above]{$\tilde{g}$} (S0);
    \draw[->] (S0) -- node[above]{$\mathrm{S}_{0}^{\varphi(2)-1}$} (S1);
    \draw[->] (S1) -- node[above]{} (Sdots);

    \draw[->] (A) -- node[below,left]{$\tilde{h}$} (R0);
    \draw[->] (A) -- node[above,right]{$\theta_0$} (R1);
    \draw[->] (R1) -- node[above,right]{$\eta_{0}$} (S1);
    \draw[->] (S1) -- node[above,right]{$\theta_{1}$} (R3);
    }
    \end{tikzpicture}
    $$
    
    It thus follows from Lemma \ref{LEMMA_BackAndForth_Iso} that there is an isomorphism $x\colon F_{\lim}\to F_{\lim}$ such that, in particular,
    
    \[h = r_0\circ \tilde{h} =  r_1\circ\mathrm{R}_0^1\circ\tilde{h}= r_1\circ \theta_0   = x\circ s_{\varphi(2)-1}\circ\eta_0 \circ \theta_0 = x\circ s_{\varphi(2)-1}\circ \mathrm{S}_{0}^{\varphi(2)-1}\circ \tilde{g} = x\circ g,\]
    which concludes the proof.
\end{myproof}

Back to the games, we want to show:

\begin{thm}\label{THM_GamesTightSqueeze}
     The pair $(\emb, \FinGame)$ has the tight squeeze property in $\Games_\A$.
\end{thm}

In order to do this, let us consider the following common notion from category theory:

\begin{defn}\label{DEFN_collectively}
    Suppose $\mathbf{C}$ is a category with arbitrary coproducts and let $\mC$ be a class of morphisms in $\mathbf{C}$. We say that a family $\mF = \set{A_i\stackrel{f_i}\to X: i\in I}$ of morphisms in $\mathbf{C}$ is {\em jointly $\mC$} if the morphism $f\colon \coprod_{i\in I}G_i\to G$ given by the universal property of the coproduct inclusions applied to $\mF$ is in $\mC$.
\end{defn}

Let
\begin{align*}
    \mathcal E &= \set{(T,A)\stackrel{f}\to (T',A'): T' = f[T]}\\
    \mathcal E^* &= \set{(T,A)\stackrel{f}\to (T',A'): f\in \mathcal E, \, A' = \overline{f}[A]},
\end{align*}
be families of $\A$-morphisms (the notation is taken from \cite{Duzi2024}). As seen in Proposition 3.3 of \cite{Duzi2024}, $\mathcal E$ is the class of epimorphisms in $\Games_\A$.

Then it follows directly from Lemma \ref{LEMMA_GameASeqLim} that
 
\begin{lemma}\label{LEMMA_GameCollectivelyE*}
    Let $\mathrm{S}\colon \omega \to \Games_{\A}$ be a sequence of injective maps. Then a cocone $(g_n\colon \mathrm{S}n\to G)_{n<\omega}$ is a colimit in $\Games_\A$ if, and only if, it is comprised of injective maps and it is jointly $\mathcal{E}^*$.
\end{lemma}
\begin{proof}
    Note that the colimit cocone used in Lemma \ref{LEMMA_GameASeqLim} has the desired properties when $\mathrm{S}$ is a sequence of injective maps. Hence, the result follows from the observation that the composition of injective maps with  isomorphisms in $\Games_\A$ (which are bijective $\A$ and $\B$-morphisms) are injective maps and that $g\circ \mF = \set{g\circ f: f\in \mF}$ is jointly $\mathcal E^*$ whenever $\mF$ is and $g$ is an isomorphism.
\end{proof}

\begin{remark}
    Four distinct orthogonal factorization systems were presented for $\Games_\A$ in \cite{Duzi2024}, with the pair $(\mathcal E^*,\mathcal M)$, where $\mathcal M$ denotes the class of injective $\A$-morphisms in $\Games_\A$, among them. In this case, Lemma \ref{LEMMA_GameCollectivelyE*} shows that $(\mathcal E^*,\mathcal M)$ is particularly ``special'' in the sense that the class $\mathcal E^*$ also complements $\mathcal M$ for the colimit of $\mathcal M$-sequences --  it can also be shown that a simple dualization of Definition \ref{DEFN_collectively} also yields a dualization of Lemma \ref{LEMMA_GameCollectivelyE*}. 

    This property is not shared by any of the other three factorization systems of $\Games_\A$ presented in \cite{Duzi2024}. Indeed, consider $(\mathcal E,\mathcal M_*)$, where $\mathcal M_*$ denotes the class of game embeddings, $\mathrm{F}_{\A}\colon \omega \to \mathbf{Fin}\A\Games_{\emb}$ be the sequence obtained in Theorem \ref{THM_AFraisseGame} and fix a colimit cocone $(f_n\colon \mathrm{F}_{\A}n\to G_{\FL})_{n<\omega}$ in $\Games_\A$. Let $G = (\bigcup_{n<\omega}\omega^n, \omega^\omega)$. Then $(f_n\colon \mathrm{F}_{\A}n\to G)_{n<\omega}$ is also a cocone of game embeddings which is jointly $\mathcal E$, despite not existing any $\A$-morphism $f\colon G\to G_{\FL}$ through which $(f_n\colon \mathrm{F}_{\A}n\to G_{\FL})_{n<\omega}$ factors. 
\end{remark}

\begin{myproof}{Theorem}{THM_GamesTightSqueeze}
    Suppose $\mathrm{S}\colon \omega \to \mathbf{Fin}\Games_{\emb}$ with cocone $(x_n\colon \mathrm{S}n\to G)_{n<\omega}$ and natural transformation $\theta\colon \mathrm{S'}\to \mathrm{S}$ in $\Games_{\emb}$  are such that $(x_n\circ \theta_n \colon S'_n\to G)_{n<\omega}$ is a colimit cocone in $\Games_\A$.

    Since $(x_n\colon \mathrm{S}n\to G)_{n<\omega}$ is a collection of game embeddings, it suffices to show in view of Lemma \ref{LEMMA_GameCollectivelyE*} that it is also jointly $\mathcal E^*$. Suppose $G=(T,A)$, $\mathrm{S}n = (T_n,A_n)$ and $\mathrm{S}n' = (T_n',A_n')$ for each $n<\omega$. Given $t\in T$, since $(x_n\circ \theta_n \colon S'_n\to G)_{n<\omega}$ is jointly $\mathcal E^*$ (we use the other implication of Lemma \ref{LEMMA_GameCollectivelyE*} here) there is an $n\in<\omega$ and an $s\in T_n'$ such that $x_n\circ \theta_n (s) = t$. Then $\theta_n(s)$ attests that $(x_n\colon \mathrm{S}n\to G)_{n<\omega}$ is jointly $\mathcal E$ for $t$. Finally, let $R\in A$. Once again, since $(x_n\circ \theta_n \colon S'_n\to G)_{n<\omega}$ is jointly $\mathcal E^*$, there is an $n\in<\omega$ and an $R'\in A_n'$ such that $\overline{x_n\circ \theta_n}(R') = \overline{x_n}\circ \overline{\theta_n}(R') = R$. Then $\overline{\theta_n}(R')$ attests that $(x_n\colon \mathrm{S}n\to G)_{n<\omega}$ is jointly $\mathcal E^*$ for $R$. 

    Now suppose that, furthermore, $\mathrm{S}'$ is a Fra\"iss\'e sequence in $\mathbf{Fin}\Games_{\emb}$. In order to show that $\mathrm{S}$ is a Fra\"iss\'e sequence as well, it suffices to show that for every $f\colon \mathrm{S}n\to G_0$ there is a $g\colon G_0\to \mathrm{S}n$ for some $m>n$ such that $g\circ f = \mathrm{S}_n^m$. We will provide the proof for the case in which the game tree $T_0$ of $G_0$ as a single branch which is not on the image of $f$ (the proof of the general case can then be obtained by a simple induction on the number of such branches), so let $R\in \Run(T_0)$ be said branch and $t\in T_0$ be the last moment of $R$ which is in the image of $f$. 
    
    Since $\mathrm{S}'$ is a Fra\'iss\'e sequence and $\mathrm{S}$ is squeezed between $\mathrm{S}'$ and the Fra\"iss\'e limit $G$, it follows that there must be an $m>n$ such that $\theta_m\colon \mathrm{S}'m\to \mathrm{S}m$ contains $\mathrm{S}_n^m(t)$ in its image and such that there is an $R'\in \Run(\mathrm{S}'m)$ extending $\theta_m^{-1}(\mathrm{S}_n^m(t))$ with the same winning condition as $R$ (that is, such that $\ali$ wins in $R'$ if, and only if, $\ali$ wins in $R$). Thus, we can let $g\colon G_0\to \mathrm{S}m$ be such that $g(s) = \mathrm{S}_n^m$ for every $s$ in the image of $f$ and $\overline{g}(R) = \overline{\theta_m}(R')$, which concludes the proof.
\end{myproof}

Thus, we have now provided a categorical proof of the previously shown:

\begin{cor}
    $G_{\FL}$ is ultrahomogeneous w.r.t. $(\emb, \mathbf{Fin}\Games)$ in $\Games_\A$.
\end{cor}

And it should be emphasized here that, while the proof provided in Section \ref{SEC_Ultrahom} for the above Corollary may be regarded as ``\emph{more straightfoward}'', it can also be seen as ``\emph{more accidental}'' -- since it relies a lot on the particularities of the objects and embeddings of $\Games_\A$. With the categorical proof provided in this section, on the other hand, we are able to reach deeper into the properties of the category $\Games_\A$ which can be generalized to other categories in order to guarantee the ultrahomogeneity of their Fra\"iss\'e limits.

We have already explained the importance of considering a proper subclass $\mM$ of $\bC$, instead of letting $\mathbf{C}_\mM=\mathbf{C}$ in Section \ref{SEC_WeaklyFiniteSmall}, but we provide here yet another argument to the significance of the further generality of the proper subclass of morphisms in our context.

Indeed, this is made clear by looking at Theorem \ref{THM_GamesTightSqueeze}: replacing $\Games_\A$ by $\Games_{\emb}$ in it does not work, since the pair $(\mathrm{Morph}(\Games_\A),\mathbf{Fin}\Games)$ fails to have the tight squeeze property in $\Games_{\A}$, as illustrated in the example below.

\begin{ex}
    Consider 
    \begin{align*}
        T_1 = \set{\seq{0:k< n}:n<\omega},  \quad &G_1 = (T_1, \emptyset)\\
        T_2 = \set{\seq{i:k<n}:i\in\{0,1\}, \, n<\omega} , \quad &G_2 = (T_2, \emptyset)
    \end{align*}
    Then $G_1$ and $G_2$ are finite games. Let $\mathrm{S}_1\colon \omega\to \mathbf{Fin}\Games_{\emb}$ and $\mathrm{S}_2\colon \omega\to \mathbf{Fin}\Games_{\emb}$ be such that $\mathrm{S}_1n = G_1$ and $\mathrm{S}_2n = G_2$ for every $n<\omega$ (with the identity being the connecting maps).

    In this case, if we consider the unique chronological $f\colon T_2\to T_1$, then $\mathrm{S}_2$ with the cocone $(f\colon \mathrm{S}_2n\to G_1)_{n<\omega}$ in $\Games_\A$ is squeezed between $\mathrm{S}_1$ and $G_1$. However, $(f\colon \mathrm{S}_2n\to G_1)_{n<\omega}$ in $\Games_\A$ is clearly not a colimit cocone of $\mathrm{S}_2$ in $\Games_\A$.
\end{ex}

\section{A brief study of the KPT-correspondence over finite games}\label{SEC_KPT}
The \textit{KPT-correspondence} established in \cite{Kechris2005} relates properties of the topological automorphisms group of the Fra\"iss\'e limit of an $L$-structure to combinatorial properties of the class of finitely generated substructures.  We end this subsection by fully describing $\Aut_{\Games}(G_{\FL})$ in terms of a known group (we don't specify here $\Games_\A$ or $\Games_\B$ because the automorphism group is the same in both categories).

Let 
$S=\{0\}\cup\set{1/n: n>0}$ 
and denote by $\mU_S$ the class of all finite ultrametric spaces whose distance function's range is contained in $S$. 

Then $\mU_S$ is a Fra\"iss\'e class in the classical model theoretical sense. In fact, such classes, as well as its Fra\"iss\'e limit and its automorphism group, were studied in depth in \cite{NguyenVanThe2009}.

Note that $\Met(G_{\FL})=((\omega^\omega,d), c_{00}(\omega))$, where
\[
d(R,R')=\begin{cases}
	\frac{1}{\Delta(R,R')+1}\text{ if $R\neq R'$,}\\
	0 \text{ otherwise.}
\end{cases}
\]

Then it can easily be shown that $\FL(\mU_S)$ is isomometric to $c_{00}(\omega)$ with the induced subspace metric, so we may consider $\FL(\mU_S)=c_{00}(\omega)$. 

On the other hand, $(\omega^\omega,d)$ is the completion of the subspace $c_{00}(\omega)$. In this case, any isometry $f$ from $c_{00}(\omega)$ into itself can be uniquely extended to an isometry $\tilde{f}$ from $\omega^\omega$ into itself. Moreover, note that $f\in \Aut_{\Sub{\CUMet}}((\omega^\omega,d), c_{00}(\omega)))$ if, and only if, $f$ is an isometry from $\omega^\omega$ into itself that restricts to an isometry from $c_{00}(\omega)$ into itself. Hence, in view of Theorem 7.3 in \cite{Duzi2024}:

\begin{prop}\label{PROP_AUT_MET}
	The group $\Aut_{\Games}(G_{\FL})$ is isomorphic to $\Aut_{\UMet}(\FL(\mU_s))$. Namely, the map 
	\begin{center}
		\begin{tikzcd}[row sep = 0em]
			\Aut_{\Games}(G_{\FL}) \arrow[r, "\varphi"] &	\Aut_{\UMet}(\FL(\mU_s)) \\
			f  \arrow[r, mapsto] &   \overline{f}|_{c_{00}(\omega)}
		\end{tikzcd}
	\end{center}
	is a group isomorphism.
\end{prop}

However, the so called KPT-correspondence explores the properties of the  \emph{topological} automorphisms group of the Fra\"iss\'e limit and we have not yet imbued $\Aut_{\Games}(G_{\FL})$ with any topology. 

In view of Proposition \ref{PROP_AUT_MET}, we can induce a topology $\tau_{\varphi}$ over $\Aut_{\Games}(G_{\FL})$ by the group isomorphism $\varphi\colon \Aut_{\Games}(G_{\FL})\to \Aut_{\UMet}(\FL(\mU_s))$, considering that $\Aut_{\UMet}(\FL(\mU_s))$ is taken with the classical topology of pointwise convergence.

Recall that a topological group $G$ is said to be \emph{extremely amenable} if every continuous action $G\curvearrowright X$ over a compact space $X$ has a fixed point. Model-theoretic KPT-correspondence tells us that the automorphisms group of a Fra\"iss\'e limit is extremely amenable if, and only if, its Fra\"iss\'e class is \emph{Ramsey} (see \cite{Kechris2005}). As we are not within the model-theoretic context, though, we need to rely on some generalizations from \cite{Bartos2024}:

\begin{defn}\label{DEF_KPT_Matching}[Definition 2.1 in \cite{Bartos2024}]
    Suppose $\mathbf F$ is a full subcategory of $\bC$ and $\mM$ is a class of $\bC$-morphisms. Let $\mathrm{S}\colon \omega\to \bF_\mM$ be a sequence, and $s= \seq{s_n\colon \mathrm{S}n\to X}$ a cocone. We say that the pair $(\mathrm{S}, s)$ is \emph{matching} if it satisfies the following conditions:
    \begin{itemize}
        \item[(F1)] For every $f\colon A\to X$ in $\mM$ with $A$ in $\bF$, there are an $n<\omega$ an $f'\colon A\to \mathrm{S}n$ such that $f = s_n\circ f'$;
        \item[(F2)] For all $n<\omega$ and $f,f'\colon A\to \mathrm{S}n$ in $\bC$ such that $s_n\circ f = s_n\circ f'$, there is an $m\ge n$ such that $\mathrm{S}^m_n\circ f = \mathrm{S}^m_n\circ f'$;
        \item[(BF)] Every back-and-forth between $\mathrm{S}$ and itself induces an isomorphism $x\colon X\to X$ which commutes with the cocone $\seq{s_n\colon \mathrm{S}n\to X}$ and the back-and-forth;
        \item[(H)] For every $h\in \Aut(X)\setminus\{\id_{X}\}$ there is an $n<\omega$ such that $h\circ s_n\neq s_n$.
    \end{itemize}
\end{defn}

\begin{prop}
    If $\mathrm{F}\colon \omega\to \mathrm{Free}[\mathbf{Fin}\Gmes]_{\emb}$ is a Fra\"iss\'e sequence in $\mathrm{Free}[\mathbf{Fin}\Gmes]_{\emb}$ and $f = \seq{f_n\colon \mathrm{F}n\to G_{\FL}}$ is its colimit cocone in $\Games_\A$, then $(\mathrm{F}, f)$ is matching.
\end{prop}
\begin{proof}
    Condition (F1) follows from Theorem \ref{THM_WeakFinSmallGames_Emb}, while the fact that each $f_n$ is monic implies (F2).

    Lemma \ref{LEMMA_BackAndForth_Iso} implies (BF) and (H) is clearly a consequence of Lemma \ref{LEMMA_GameCollectivelyE*}.
\end{proof}

\begin{defn}[Definition 3.1 in \cite{Bartos2024}]
    Suppose $\mathbf F$ is a full subcategory of $\bC$ and $\mM$ is a class of $\bC$-morphisms. We say that $\bF_\mM$ has the \emph{weak Ramsey property} is for every object $A$ in $\bF$ there is an $a\colon A\to A'$ in $\bF_\mM$ such that for all $k<\omega$, $B$ object in $\bF$ and finite $\mF\subset \mM(A,B)$ there is an object $C$ in $\bF$ such that for every (``\emph{colouring}'') $\varphi\colon \mM(A,C)\to k$ there is a $c\colon B\to C$ such that $\varphi$ is constant (or ``\emph{monochromatic}'') on $c\circ \mF$.
\end{defn}

In what follows, \cite{Bartos2024} considers the topology over $\Aut(X)$ for a matching pair $(\mathrm{S}, s)$ defined with the neighborhood base of the identity given by
\[
\set{\set{g\in \Aut(X): g\circ s_n = s_n} :n<\omega}.
\]

Its is then easy to check that, over $\Aut_{\Games}(G_{\FL})$, such topology produces the aforementioned $\tau_{\varphi}$. This is relevant for us because:

\begin{thm}\label{THM_Ramsey_ExtremeAmenable}[Theorem 3.14 in \cite{Bartos2024}]
    Suppose an object $F$ in $\bC$ is the colimit in $\bC$ of a Fra\"iss\'e sequence $\mathrm{F}\colon \omega\to \bF_\mM$ in $\bF_\mM$ with colimit cocone $f = (f_n)_{n<\omega}$ such that $(\mathrm{F},f)$ is matching. Then the following are equivalent:
    \begin{itemize}
        \item[(a)] $\Aut(F)$ is extremely amenable;
        \item[(b)] $\bF_\mM$ has the weak Ramsey property.
    \end{itemize}
\end{thm}
    
In this case, as it is a consequence of Theorem 11 in \cite{NguyenVanThe2009} that $\Aut_{\UMet}(\FL(\mU_s))$ is not extremely amenable, we obtain that

\begin{cor}
    The topological group $(\Aut_{\Games}(G_{\FL}),\tau_{\varphi})$ is not extremely amenable. Furthermore, $\mathrm{Free}[\mathbf{Fin}\Gmes]_{\emb}$ does not have the weak Ramsey property.
\end{cor}

Note that, in order to rely on Theorem \ref{THM_Ramsey_ExtremeAmenable}, we crucially relied on the weak finite smallness of the games in $\mathrm{Free}[\mathbf{Fin}\Gmes]$. The topology $\tau_{\mathrm{p}}$ over $\Aut_{\Games}(G_{\FL})$ induced by the pointwise convergence topology of $\Aut_{\UMet}(\omega^\omega)$ is yet another possibility which seems to be better suited to explore the combinatorics of $\mathbf{FinGame}_{\emb}$ (as its basic open sets can also fix runs which are won by $\bob$, whereas $\tau_\phi$'s basic open sets only restricts runs which are won by $\ali$). However, since (F1) of Definition \ref{DEF_KPT_Matching} fails for the Fra\"iss\'e sequence in $\FinGame_{\emb}$ and its colimit cocone in $\Games_\A$, we cannot rely on the work done in \cite{Bartos2024} to relate the combinatorics of $\FinGame_{\emb}$ to the topological dynamics of $\Aut_{\Games}(G_{\FL})$. The following thus remains open:

\begin{que}\label{QUE_FinGame_Ramsey}
    Does $\FinGame_{\emb}$ have the weak Ramsey property? And does the answer to this question bear any implications to the topological dynamics of $(\Aut_{\Games}(G_{\FL}), \tau_{\mathrm{p}})$?
\end{que}

We expect that further generalizations of the work done in \cite{Bartos2024} which do not rely on (F1) are needed to answer Question \ref{QUE_FinGame_Ramsey}.

\bibliographystyle{plainurl}

\end{document}